\documentclass{article}
\usepackage[utf8]{inputenc}
\usepackage[algonl,boxed,norelsize,lined]{algorithm2e}
\usepackage{amscd,amsmath,amsthm,amssymb,amsfonts}
\usepackage{fullpage}
\usepackage{comment}
\usepackage[all]{xy}
\usepackage{datetime}
\usepackage{color}
\usepackage[english]{babel}
\usepackage[T1]{fontenc}
\usepackage{hyperref}
\usepackage{enumitem}
\usepackage{theoremref}

\usepackage{booktabs}

\usepackage[a4paper,top=2.5cm,bottom=2cm,left=2.5cm,right=2.5cm,marginparwidth=1.75cm]{geometry}

\usepackage{amsmath}
\usepackage{amssymb}
\usepackage{amsthm}
\usepackage{verbatim}

\usepackage{tikz}
\usetikzlibrary{decorations.pathreplacing}
\usetikzlibrary{calc}
\usetikzlibrary{patterns}

\newlength{\bibitemsep}
\newlength{\bibparskip}
\let\oldthebibliography\thebibliography
\renewcommand\thebibliography[1]{%
  \oldthebibliography{#1}%
  \setlength{\parskip}{\bibitemsep}%
  \setlength{\itemsep}{\bibparskip}%
}

\newtheorem{Theorem}{Theorem}[section]

\newtheorem{Lemma}[Theorem]{Lemma}
\newtheorem{Corollary}[Theorem]{Corollary}
\newtheorem{Proposition}[Theorem]{Proposition}
\newtheorem{Remark}[Theorem]{Remark}
\newtheorem{Definition}[Theorem]{Definition}
\newtheorem{Example}[Theorem]{Example}

\newtheorem{Conjecture}[Theorem]{Conjecture}

\numberwithin{equation}{section}

\DeclareMathOperator{\init}{in}
\DeclareMathOperator{\prin}{prin}
\DeclareMathOperator{\Gr}{Gr}
\DeclareMathOperator{\fr}{fr}
\DeclareMathOperator{\GF}{GF}
\DeclareMathOperator{\GR}{GR}
\DeclareMathOperator{\Mu}{M}
\DeclareMathOperator{\gr}{gr}

\DeclareMathOperator{\up}{up}
\DeclareMathOperator{\val}{val}
\DeclareMathOperator{\trop}{Trop}

\DeclareMathOperator{\spec}{Spec}
\DeclareMathOperator{\proj}{Proj}

\newcommand{\gv}{\mathbf{g}}

\newcommand{\RR}{\mathbb{R}}
\newcommand{\BB}{\mathbb{B}}

\DeclareMathOperator{\Flag}{\mathcal{F}\hspace{-1.6pt}\ell}

\definecolor{caribbeangreen}{rgb}{0.0, 0.8, 0.6}
\definecolor{capri}{rgb}{0.0, 0.75, 1.0}

\title{Tropical totally positive cluster varieties}
\author{Lara Bossinger\\
\small  \url{lara@im.unam.mx}\\
\small 
Universidad Nacional Autónoma de México\\
\small  Instituto de Matemáticas Unidad Oaxaca,
\small  León 2 altos, 68020 Oaxaca, México\\
}

\begin{document}

\maketitle

\begin{abstract}
We study the relation between the integer tropical points of a cluster variety (satisfying the full Fock--Goncharov conjecture) and the totally positive part of the tropicalization of an ideal presenting the corresponding cluster algebra.
Suppose we are given a presentation of the cluster algebra by a Khovanskii basis for a collection of $\gv$-vector valuations associated with several seeds related by mutations.
In presence of a full rank fully extended exchange matrix we construct the rays of a subfan of the totally positive part of the tropicalization of the ideal that coincides combinatorially with the subgraph of the exchange graph of the cluster algebra corresponding to the collection of seeds.
Moreover, geometric information about Gross--Hacking--Keel--Kontsevich's toric degenerations associated with seeds gets identified with the Gröbner toric degenerations obtained from maximal cones in the tropicalization.
    
As application we prove a conjecture about the relation between Rietsch--Williams' valuations for Grassmannians arising from plabic graphs \cite{RW17} to Kaveh--Manon's work on valuations from the tropicalization of an ideal \cite{KM16}. 
In a second application we give a partial answer to the question if the Feigin--Fourier--Littelmann--Vinberg degeneration of the full flag variety in type $\mathtt A$ is isomorphic to a degeneration obtained from the cluster structure. 
\end{abstract}

\section{Introduction}

Cluster varieties are the geometric analogues of Fomin--Zelevinyky's cluster algebras \cite{FZ02,FZ03,FZ07} that were introduced by Fock and Goncharov \cite{FG_ensem}.
For cluster algebras there is a duality between \emph{seed patterns} (consisting of cluster variables) and \emph{Y-patterns} (consisting of $y$-variables). 
Similarly cluster varieties arise in pairs: $\mathcal A$-varieties are unions of tori with a system of local coordinates given by a seed pattern, while $\mathcal X$-varieties are unions of tori with a system of local coordinates given by a $Y$-pattern.

In this paper, we are interested in projective varieties $Y$ that are minimal models of an $\mathcal A$-variety. 
We assume the corresponding seed pattern has frozen directions and no coefficients\footnote{Following \cite{BFMN} we systematically distinguish these two notions and refer to page 2 in {\it loc.cit.} for an explanation.}. 
Moreover, (up to codimension 2) the {\it boundary}, {\it i.e.} the difference between the projective variety $Y$ and the cluster variety $\mathcal A$, should correspond to the vanishing locus of the frozen variables (as in \cite[Construction B.1]{GHKK14}).
The cluster algebra is a subalgebra of the ring of regular functions on $\mathcal A$ by the Laurent phenomenon \cite{FZ02}.
In this setting there are (at least) two ways of tropicalizing $Y$:
\begin{itemize}
    \item[T1] The {\bf tropicalization of an ideal} $J$ presenting $Y$ in a given embedding, which is a distinguished polyhedral subfan of Mora--Robbiano's Gröbner fan of $J$ \cite{MoraRobbiano}, see \S~\ref{sec:trop}; for $J\subset k[x_1,\dots,x_N]$ homogeneous the Gröbner fan is a complete fan in $\mathbb R^N$ whose maximal cones correspond to monomial initial ideals obtained from $J$. The tropicalization is a subfan of dimension $\dim Y+1$.
    \item[T2] The {\bf Fock--Goncharov tropicalization} of the (compactified) cluster variety which corresponds to the rational points of $Y$ over a given semifield, for example the integers with operations addition and maximum, denoted $\mathbb Z^T$, see \S~\ref{sec:p*}.
\end{itemize}
Total positivity plays a role in both types of tropicalization: if $J$ is a homogeneous ideal in $\mathbb R[x_1,\dots,x_N]$ the tropicalization $\trop(J)$ contains a closed subfan known as its {\bf totally positive part} (it consists of those weight vectors $w\in \mathbb R^N$ for which the initial ideal of $J$ is totally positive, see Definition~\ref{def:totally pos}).
The Fock--Goncharov tropicalization of a cluster variety is a special case of tropicalization for schemes with a {\bf positive atlas} \cite{FG_ensem}, {\it i.e.} schemes glued from patches ({\it e.g.} affine spaces or algebraic tori) whose transition functions are substraction free in the local coordinates which allows to consider rational points with respect to semifields.
The aim of this paper is to shed some light on how the two notions of total positivity and tropicalization are related. 

Throughout we assume the $\mathcal A$-variety satisfies the {\it full Fock--Goncharov conjecture} \cite[\S1]{GHKK14}, so that the associated (ordinary) cluster algebra $A$ has a basis of theta functions.
Typically $A$ is the homogeneous coordinate ring of the projective variety $Y$.
Our results may be generalized to section rings of certain line bundles compatible with the boundary $Y\setminus \mathcal A$.
More details on the relevant notion of compatibility are provided in forthcoming work \cite{BCMN_NO}.
\medskip

{\bf Cluster ensemble maps.}
In their fourth paper \cite{FZ07} Fomin and Zelevinsky show that a $Y$-pattern can be obtained from a cluster pattern by considering certain monomials in the {\it cluster variables} ({\it i.e.} elements of a cluster seed), called {\it $\hat y$-variables}.
They work without {\it frozen directions}, which are typically present in geometrically interesting classes of cluster algebras.

Before proceeding we need some notation for cluster algebras.
Let $\Gamma$ denote the fixed data necessary to determine a cluster algebra $A(\Gamma)$. 
It consists of a lattice $N\cong \mathbb Z^{n+m}$ endowed with a skew-symmetric bilinear. 
Here $n$ is the number of mutable directions and $m$ is the number of frozen ones.
Let $I$ be the index set of directions, with $I_{\rm mut}\subset I$ the subset of mutable directions, so $n=|I_{\rm mut}|$.
$N^\circ\subset N$ denotes a saturated sublattice of finite index and let $M$ resp. $M^\circ$ be the dual lattices (Definition~\ref{def:fixed_data}). 
A seed $s$ is a basis $\{e_{i;s}:i\in I\}$ for $N$ that induces a basis $\{f_{i;s}:i\in I\}$ for $M^\circ$.
Let $\tilde B_s=(b_{ij})\in \mathbb Z^{(n+m)\times n}$ be the exchange matrix of $s$.
Then the $i$th $\hat y$-variable is defined as $\hat y_{i;s}=\prod_{j=1}^{n} A_{j;s}^{b_{ji}}$ for $i\le n$ mutable, where $A_{j;s}:=z^{f_{j;s}}$.
To extend to frozen directions we need to complete $\tilde B_s$ to a {\bf fully extended exchange matrix} $\widetilde{B_s}\in \mathbb Z^{(n+m)^2}$. 
The $m\times m$-submatrix corresponding to rows and columns of frozen directions may be chosen arbitrarily.
Any matrix of form $\widetilde{B_s}$ defines a cluster ensemble lattice map $p^*:N\to M^\circ$, see \S\ref{sec:p*}.
These maps are crucial to the theory as they induce morphism between the $\mathcal A$- and $\mathcal X$-varieties.
The $\mathcal A$-variety associated with $\Gamma$, denote it $\mathcal A_\Gamma$, is glued from tori $T_{N^\circ;s'}$ while the $\mathcal X$-variety $\mathcal X_\Gamma$ is obtained from several copies of $T_{M;s'}$. 
The seed $s'$ in the index indicates that there is one torus for each seed $s'$ related to $s$ by mutation (Definition~\ref{def:mut seed data}).
The lattice map $p^*:N\to M^\circ$ correspond to the pullback of a map of tori $T_{N^\circ}\to T_M$ which extends to give a morphism $p:\mathcal A_\Gamma \to \mathcal X_\Gamma$ called a {\bf cluster ensemble map}.
A seed $s$ gives local coordinates to the cluster varieties:  $A_{i;s}=z^{f_{i;s}}$ for $\mathcal A_{\Gamma}$ and $X_{i;s}:=z^{e_{i;s}}$ for $\mathcal X_\Gamma$, $i\in I$.
Fomin--Zelevisnky's $\hat y$-variables are generalized by the pullbacks of $\mathcal X$-coordinates along the cluster ensemble map $p^*(X_{i;s})=\prod_{j=1}^{n+m}A_{j;s}^{b_{ji}}$.
\medskip

If a cluster ensemble map is given by a lattice isomorphism (so $\widetilde{B_s}$ is unimodular) then we may consider the variables $(p^*)^{-1}(A_{i;s})$.
By the positivity of the Laurent phenomenon these yield a {\it positive parametrization} of $\mathcal A_\Gamma$ and $Y$\footnote{By a positive parametrization we mean a system of coordinates for $\mathcal A_\Gamma$ (and $Y$) with the property that all cluster variables can be expressed as Laurent polynomials with non-negative coefficients in the given coordinates.} that is extremely useful in applications.
One example of such an application is Speyer and Williams paper \cite{SW05} in which they use a positive parametrization for the Grassmannian previously defined by Postnikov \cite{Pos06}. 
Fixing an initial seed with good combinatorial properties they compute the expressions of Plücker coordinates in $\hat y$-variables.
The tropicalization of each expression yields a fan in $\mathbb R^n$ ($n$ being the number of mutable directions).
Let $F$ denote the common refinement of all fans obtained this way.
The main result of their paper is that for the Grassmannian of planes $F$ coincides combinatorially with the exchange graph of the cluster algebra and with the positive part of the tropicalization of the Plücker ideal, also known as the {\it tropical Grassmannian} \cite{SS04}.
The latter has a fan structure induced by the Gröbner fan of the Plücker ideal.
Beyond the Grassmannian of planes $F$ is combinatorially different from the exchange graph and from the positive part of the tropicalization.
This paper may be seen as a far generalization of Speyer and Williams work with emphasis on the geometric data given by toric degenerations that can be associated with both, cones in the tropicalization and seeds. 

\medskip
{\bf Toric degenerations.} We use cluster ensemble maps to compute rays of (the totally positive part of) the tropicalization of an ideal presenting a cluster algebra. 
More precisely, we are interested in toric degenerations of $Y$, {\it i.e.} flat morphisms $\pi:\mathcal Y\to \mathbb A^1$ with $\pi^{-1}(1)\cong Y$ the {\bf generic fibre} and $\pi^{-1}(0)$ a projective toric variety called the {\bf central fibre}.
Toric degenerations can be obtained from either type of tropicalization:
\begin{itemize}
    \item[T1] A maximal cone in $\trop(J)$ whose initial ideal is binomial and prime defines a {\bf Gröbner toric degeneration} of $Y$; The main challenge of this approach is to compute the fan $\trop(J)$, in particular its rays are generally hard to get.
    \item[T2] For every seed $s$ of $A(\Gamma)$ there exists a full rank valuation $\gv_s:A(\Gamma)\setminus\{0\}\to \mathbb Z^{n+m}$ with finitely generated value semigroup that defines a toric degeneration of $Y$ in the sense of \cite{An13,GHKK14}. The {\bf Newton--Okounkov polytope} of the valuation defines the projective toric variety that is the central fiber of the degeneration (more details in \S\ref{sec:val}).
\end{itemize} 

To be more precise, we need to consider the fixed data $\Gamma^\vee$ that is {\it Langlands dual} to $\Gamma$, see Definition~\ref{def:Langlands dual}.
If the full Fock--Goncharov conjecture holds for $\mathcal A_\Gamma$ then the Fock--Goncharov tropicalization of $\mathcal X_{\Gamma^\vee}$, denoted $\mathcal X_{\Gamma^\vee}(\mathbb Z^T)$, indexes a theta basis $\{\vartheta_q:q\in\mathcal X_{\Gamma^\vee}(\mathbb Z^T)\}$ of regular functions on $\mathcal A_\Gamma$.
Notice that $A(\Gamma)$ is contained in the ring of regular functions on $\mathcal A_\Gamma$ by the Laurent phenomenon.
Given a seed $s$ we may identify $\mathcal X_{\Gamma^\vee}(\mathbb Z^T)\equiv \mathbb Z^{n+m}$ and the association $\vartheta_q\overset{s}{\leftrightarrow} q\in \mathbb Z^{n+m}$ induces the valuation $\gv_s:A(\Gamma)\setminus \{0\}\to \mathbb Z^{n+m}$. 
The images of cluster variables under the valuation $\gv_s$ are precisely Fomin--Zelevinsky's $\gv$-vectors. 
The geometric information ({\it i.e.} the Newton--Okounkov polytope and the toric degeneration) of the valuation is captured by a {\bf Khovanskii basis}, {\it i.e.} a set of algebra generators $b_1,\dots,b_r\in A(\Gamma)$ such that $\gv_s(b_1),\dots,\gv_s(b_r)$ generate the image of $\gv_s$ as a semigroup.
Given a seed $s$ and a Khovanskii basis $\mathcal B$ for $\gv_s$ we denote by $G_{s;\mathcal B}$ the matrix whose columns are $\gv_s(b)$ for $b\in \mathcal B$.
A Khovanskii basis further determines a presentation of $A(\Gamma)\cong k[x_1,\dots,x_r]/J_{\mathcal B}$ induced by $b_i\mapsto x_i$.
Our main result can now be stated as follows.

\begin{Theorem}\label{thm:main}
Fix an initial seed $s$ and a set $\mathcal B:=\{A_1,\dots,A_N\}$ of cluster variables in $A(\Gamma)$ containing the cluster variables $A_{i;s}$ and $A_{i;\mu_k(s)}$ for all $i\in I$ and $k\in I_{\rm mut}$.
Assume  $\mathcal B$ is a Khovanskii basis for $\gv_s$ and all $\gv_{\mu_k(s)}, k\in I_{\rm mut}$ simultaneously.
Then for every full rank fully extended exchange matrix $\widetilde{B_s}$ the rows of the matrix
\[
-\widetilde{B_s}^{-T}G_{s;\mathcal B}
\]
corresponding to mutable directions 
are rays of a maximal prime cone in the totally positive part of the tropicalization $\trop^+(J_{\mathcal B})$. 
\end{Theorem}

The following statement of the abstract is a consequence.

\begin{Corollary}
Let $s=s_0,s_1,\dots,s_r$ be a collection of seeds related by mutation. 
Fix a Khovanskii basis $\mathcal B:=\{A_1,\dots,A_s\}$ for all $\gv_{s_i}$ simultaneously containing the cluster variables of $s,s_1,\dots,s_r$. 
Given a fully extended exchange matrix of full rank $\widetilde{B_s}$ the columns of the matrix
\[
-\widetilde{B_{s_j}}^{-T}G_{s_j;\mathcal B}
\]
are rays of a maximal prime cone $\tau_{j}\in\trop^+(J_{\mathcal B})$ for every $j$. 
The subfan of $\trop^+(J_{\mathcal B})$ generated by $\{\tau_{j}:j=0,\dots,r\}$ is combinatorially equivalent to the subgraph of the exchange graph of $A(\Gamma)$ given by $\{s,s_1,\dots,s_r\}$.
Moreover, the generic fibres of the toric degenerations of $Y$ induced by $\tau_{j}$ and $s_j$ are isomorphic for all $j$.
\end{Corollary}

Assume $A(\Gamma)$ satisfies the assumptions of Theorem~\ref{thm:main} and has a {\bf finite global Khovanskii basis} $\mathcal B^{\text{gl}}$, {\it i.e.} a set of algebra generators that is a Khovanskii basis for {\it all} valuations $G_{s;\mathcal B}$ ($s$ any seed) simultaneously. 
We conjecture:

\begin{Conjecture}\label{conjecture}
If $A(\Gamma)$ has a finite global Khovanskii basis then $A(\Gamma)$ is of finite cluster type.
\end{Conjecture}

A strategy to prove the conjecture would be to strengthen Theorem~\ref{thm:main} to relax the assumption that all cluster variables of the seeds $s$ and $\mu_k(s)$ for all $k$ have to be contained in the Khovanskii basis determining $J_{\mathcal B^{\text{gl}}}$.
Then an identification of (an infinite number of) seeds with cones in $\trop^+(J_{\mathcal B^{\text{gl}}})$ (which is a finite fan) would yield a contradiction.

\subsection{Applications}

{\bf Grassmannians.} 
In \cite{KM16} Kaveh--Manon construct valuations $\nu_\sigma:A\setminus\{0\}\to \mathbb Z^d$ from maximal prime cones $\sigma$ in the tropicalization of an ideal $J$ presenting $A$ (see \S\ref{sec:val from trop} for a brief summary). 
If two maximal prime cones $\sigma, \sigma'$ share a facet then $\nu_\sigma$ and $\nu_{\sigma'}$ may be constructed in such a way that for any nonzero $f\in A$ the valuations $\nu_\sigma(f)$ and $\nu_{\sigma'}(f)$ differ in a single entry.
Let $A_{k,n}$ be cluster algebra that is the homogeneous coordinate ring of the Grassmannian $\Gr(k,\mathbb C^n)$ with respect to its Plücker embedding \cite{Sco06}.
Rietsch and Williams \cite{RW17} construct valuations $\val_s:A_{k,n}\setminus\{0\}\to \mathbb Z^{k(n-k)}$ from seeds $s$ in $A_{k,n}$.
Using the combinatorics of plabic graphs \cite{Pos06} it is shown that these valuations satisfy for any given $f\in A_{k,n}$ that $\val_s(f)$ and $\val_{s'}(f)$ differ only in one entry if $s$ and $s'$ are related by mutation.
In the context of \cite{BFFHL} the conjecture arose that Rietsch--Williams' valuations are a special case of Kaveh--Manon's valuations. 
We affirm this conjecture with the following result.

\begin{Theorem}\label{thm:Gr plabic trop}
Given a seed $s$ for $\Gr(k,\mathbb C^n)$ whose associated Newton--Okounkov polytope is integral, let $M_s$ be the matrix whose columns are $\{\val_s(p_J):J\in\binom{n}{k}\}$. 
Then the rows of $M_s$ are rays in the totally positive tropical Grassmannian $\trop^+(J_{k,n})$, where $J_{k,n}$ is the Plücker ideal.
\end{Theorem}

{\bf Flag varieties.}
As a second application we study toric degenerations of the full flag variety of type $\mathtt A$, so here $J=J_n$ is the corresponding Plücker ideal.
A famous toric degeneration is given by the Feigin--Fourier--Littelmann--Vinberg polytope \cite{FFL_PBWtypeA} that is a Newton--Okounkov polytope for a specific valuation \cite{Kiri_FFLV=NO} ({\it FFLV degeneration}, for short).  
In \cite{GHKK14} a class of toric degenerations that can be obtained from seeds in the cluster algebra associated with the flag variety \cite{BFZ05} ({\it cluster degenerations}, for short) is introduced. 
Peter Littelmann posed the question if the FFLV degeneration is isomorphic to a cluster degeneration.
More precisely, {\it isomorphism} refers to a unimodular equivalence of the associated Newton--Okounkov polytopes.
For the FFLV degeneration a maximal prime cone in the tropical flag variety is known \cite{FFFM_PBW_tropFlag}. 
The main result of this paper identifies those cluster degenerations for which the Plücker coordinates constitute a Khovanskii basis with maximal prime cones in the totally positive tropical flag variety $\trop^+(J_n)$.
Symmetries of the tropical flag variety induce isomorphisms between the Newton--Okounkov polytopes associated with maximal prime cones that lie in the same orbit. 
We show that the known symmetries induced by the action of the symmetric group (and $\mathbb Z_2$) on the tropical flag variety do not yield the desired isomorphism between the FFLV and the cluster degenerations. 
More precisely, the orbit of the maximal prime cone associated with the FFLV degeneration does not intersect the totally positive part.

\begin{Corollary}\label{cor:FFLV not gv sym}
There is no unimodular equivalence between the FFLV-polytope and the Newton--Okounkov polytope of a $\gv$-vector valuation that is induced by a symmetric group element.
\end{Corollary}

The literature on flag varieties in the context of Luztig's total positivity \cite{Lus_flag_semifield} is vast and it would certainly be interesting to explore the connection to the present paper.
More recently the tropical geometry of flag varieties has been studied \cite{BLMM,BEZ_trop_flag}. 
Both approaches are combined in work of Boretsky and collaborators on positive tropical flags where a notion of flag positroids is introduced \cite{Jonathan_trop_pos_flags,BEW}.

\subsection{Structure of the paper}
In \S\ref{sec:init} we recall background on initial ideals, tropicalization of ideals and its totally positive part \S\ref{sec:trop} as well as on valuations \S\ref{sec:val} and their relation to the tropicalization \S\ref{sec:trop from val}-\ref{sec:val from trop}. 
In \S\ref{sec:cluster} we recall background on cluster varieties and algebras \S\ref{sec:ca and cv}, on principal coefficients \S\ref{sec:prin} and cluster ensemble maps \S\ref{sec:p*} as well as $\gv$-vector valuations \S\ref{sec:gv} before proving the main result in \S\ref{sec:main}.
The application to the Grassmannian is presented in \S\ref{sec:Grass} and the application to flag varieties in \S\ref{sec:flag}.

\bigskip
\noindent
{\bf Acknowledgements.} 
The manuscript was finalized during a research visit at the University of Cologne, in particular I would like to thank Peter Littelmann and Bea Schumann for their hospitality and the opportunity to present this work in their research seminar. 
Further, I would like to thank Chris Eur and Lauren Williams for helpful discussions and explaining their results to me.
I was partially supported by the PAPIIT project IA100122, Dirección General de Asuntos del Personal Académico, Universidad Nacional Autónoma de México 2022. 

\section{Initial ideals and valuations}\label{sec:init}
Let $k$ be a field.
For $m\in \mathbb Z^N_{\ge 0}$ we write $x^m:=x_1^{m_1}\cdots x_n^{m_N}\in k[x_1,\dots,x_N]$.
A total order on the set of monomials in $k[x_1,\dots,x_N]$ is a {\bf term order} if it satisfies:
\[
(i)\  1 < x^m\  \forall m\in \mathbb Z^n_{\ge 0} \setminus \{0\} \quad \text{and} \quad (ii)\  x^a<x^{b} \Rightarrow x^{a+c}<x^{b+c}\ \forall a,b,c\in \mathbb Z^N_{\ge 0}.
\]
The {\bf leading term} of an element $f=\sum c_ax^a\in k[x_1,\dots,x_N]$ with respect to a term order $<$ is $\init_<(f)=c_bx^b:=\max_<\{c_ax^a:c_a\not =0\}$, where $c_b$ is the {\bf leading coefficient} and $x^b$ is the {\bf leading monomial}.
For an ideal in $J\subset k[x_1,\dots,x_N]$ we define its {\bf initial ideal with respect to $<$} as 
$\init_<(J):=(\init_<(f):f\in J)$.
The initial ideal is finitely generated and a generating set $G$ of $J$ that satisfies $(\init_<(g):g\in G)=\init_<(J)$ is called a {\bf Gr\"obner basis}.
Every ideal posses a finite number of distinct initial ideals \cite{Stu96}.
It has been shown by Mora and Robbiano that the initial ideals can be organized in a polyhedral fan \cite{MoraRobbiano}. 
To see how, we need the notion of initial ideals with respect to weight vectors:
fix $w\in \mathbb R^N$, we call it a {\bf weight vector} and define the {\bf initial form} of an element $f=\sum c_ax^a$ with respect to $w$ as 
\[
\init_w(f) =\sum_{b:\  w\cdot b=\max\{w\cdot a:c_a\not =0\}}c_bx^b.
\]
Notice that depending on $w$ and $f$ the initial form $\init_w(f)$ is not necessarily just one term.
Similarly to the above we define the {\bf initial ideal} of $J$ {\bf with respect to $w$} as $\init_w(J):=(\init_w(f):f\in J)$.
For any weight vector $w$ we may define the {\bf homogenization} of $J$ in $k[x_1,\dots,x_N,t]$: for a single element $f=\sum c_ax^a$ we set
\[
f^{h;w}:=\sum c_ax^a t^{\max\{w\cdot b:c_b\not =0\}-w\cdot a}.
\]
Similarly, for the ideal $J$ we define $J^{h;w}:=(f^{h;w}:f\in J)$. The homogenization of $J$ is a family of deformations of $J$ and the quotient algebra $A^{h;w}:=k[x_1,\dots,x_N,t]/J^{h;w}$ is a free $k[t]$-module \cite[\S15.8]{Eisenbud}.
Let $A^w:=k[x_1,\dots,x_N]/\init_w(J)$.
The degeneration of $\spec(A)$ to $\spec(A^w)$ defined by $\spec(A^{h;w})$ is called a {\bf Gröbner degeneration}.

Given the ideal $J$ any term order can be {\bf represented} by a weight vector in $w\in \mathbb Z^N_{>0}$ (see, e.g. \cite[Lemma 3.1.1]{HH_monomial}), that is $\init_w(J)=\init_<(J)$.
Conversely, a weight vector $w$ belongs to the {\bf Gr\"obner region} $GR(J)$ if there exists a term order $<$ such that $\init_<(\init_w(J))=\init_<(J)$. 
The Gr\"obner region carries a fan structure, called the {\bf Gröbner fan} $\GF(J)$ that was discovered by Mora and Robbiano in \cite{MoraRobbiano}. 
Two weight vectors $v,w\in \RR^N$ lie in the relative interior of a cone $C$, denoted by  $v,w\in C^\circ$, if and only if $\init_v(J)=\init_w(J)$.
The maximal dimensional cones in $\GF(J)$ correspond to monomial initial ideals coming from term orders.
These are particularly useful as they induce vector space bases for the quotient algebra $A=k[x_1,\dots,x_N]/J$: we call a monomial $x^a$ that is {\em not} contained in $\init_<(J)$ a {\bf standard monomial}. The set $\BB_<:=\{\bar x^a\in A:x^a\not \in \init_<(J)\}$ is a vector space basis of $A$ called a {\bf standard monomial basis}.
In fact, if $w\in C$ for some maximal cone $C\in\GF(J)$ associated to $<$, then $\BB_<$ is a basis for the free $k[t]$-module $A^{h;w}$, see e.g. \cite[Proof of Theorem 15.17]{Eisenbud}.

\subsection{Tropicalization of ideals and its totally positive part}\label{sec:trop}

The Gr\"obner fan has a subfan of interest to us. We define the {\bf tropicalization} of $J$:
\[
\trop(J):=\{w\in\GR(J): \init_w(J) \ \text{ does not contain monomials}\}.
\]
For simplicity we assume $J$ is homogeneous with respect to a positive (multi-)grading, so the tropicalization is a pure fan whose dimension coincides with the Krull-dimension of $A$.
Moreover, $\trop(J)$ and $\GF(J)$ contain a linear subspace $\mathcal L_J$, called {\bf lineality space} which consists of $w\in \mathbb R^N$ with $\init_w(J)=J$.

\begin{Lemma}\label{lem:grading and lineality}
Let $J$ be a (multi-)homogeneous ideal inside $S:=k[x_{i_j}:1\le i\le m,1\le j\le k_i]$ with respect to a $\mathbb Z^m_{\ge 0}$-grading given by $\deg(x_{i_j})=e_i$ for some $k_i$ that satisfy $k_1+\dots+k_m=N$ and $\{e_i:1\le i\le m\}$ is the standard basis of $\mathbb Z^m$.
Then for $1\le i\le m$ we have 
\begin{equation}\label{eq:elements lineality}
    \ell_i:=(0,\dots,0,1,\dots,1,0\dots,0) \in \mathcal L_J
\end{equation}
where the $1$'s appear in the positions $i_1,\dots,i_{k_i}$.
\end{Lemma}

Among the maximal cones of $\trop(J)$ we may look for {\bf prime cones} whose associated initial ideal is binomial and prime, hence toric. In particular, any Gr\"obner degeneration associated to a weight in the interior of a maximal prime cone is in fact a {\bf toric degeneration}.

\begin{Definition}\label{def:totally pos}
Let $J\subset \mathbb R[x_1,\dots,x_N]$, then $J$ is {\bf totally positive} if there does not exist any nonzero $f\in J$ with $f\in\mathbb R_{\ge 0}[x_1,\dots,x_N]$.
In this case we may define the {\bf totally positive part of $\trop(J)$} as the subfan
\[
\trop^+(J):=\{w\in \trop(I):\init_{w}(J) \text{ is totally positive}\}.
\]
\end{Definition}

This is not the original definition but equivalent by \cite[Proposition 2.2]{SW05}.
The study of tropical totally positive varieties was initiated by Speyer and Williams in \cite{SW05} where they focus in particular on the tropical totally positive Grassmannain.
\begin{Proposition}[\cite{ET_posPolynomial,H_pos_poly}]\label{prop:tot pos and vanishing}
An ideal $J\subset \mathbb R[x_1,\dots,x_N]$ is totally positive if and only if $(\mathbb R_{>0})^n\cap V(\init_w(J))\not=\varnothing$ for some $w\in \mathbb R^N$.
\end{Proposition}
In particular, $\trop(J)$ and $\trop^+(J)$ are closed subfans of $\GF(I)$.

\subsection{Valuations}\label{sec:val}

Let $k$ be an algebraically closed field. 
Throughout the paper we denote by $A$ a $k$-algebra and domain.
For simplicity we may assume that $A$ has a {\bf positive (multi-)grading} that is $A=\bigoplus_{w\in \mathbb Z^m_{\ge 0}}A_w$.
Let $\Gamma$ be an abelian group that is totally ordered by $<$.
By a {\bf (Krull) valuation} on $A$ we mean a map $\nu:A\setminus \{0\}\to \Gamma$ that satisfies for all $a,b\in A$ and $c\in k$
\[
\nu(ab)=\nu(a)+\nu(b), \quad \nu(a+b)\le \max\{\nu(a),\nu(b)\},\quad \nu(ca)=\nu(a).
\]
If $\nu$ only satisfies $\nu(ab)\le \nu(a)+\nu(b)$ it is called a {\bf quasivaluation}.
Notice that the image of a valuation $\nu$ carries the structure of an additive semigroup. It is therefore called the {\bf value semigroup} of $\nu$ and we denote it by $S(A,\nu)$.
The {\bf rank} of $\nu$ is defined as the rank of the group completion of its semigroup inside $\Gamma$, $\nu$ is said to have {\bf full rank} if its rank coincides with the Krull dimension of $A$.
Every valuation induces a filtration on $A$ with filtered pieces for $\gamma \in \Gamma$ defined by
\[
\mathcal F_{\nu,\gamma}:=\{a\in A:\nu(a)\ge \gamma \}\quad  \left(\text{resp. }\mathcal F_{\nu,>\gamma}:=\{a\in A:\nu(a)> \gamma\}\right).
\]
The {\bf associated graded algebra} is $\gr_\nu(A):=\bigoplus_{\gamma\in \Gamma} \mathcal F_{\nu,\gamma}/\mathcal F_{\nu,>\gamma}$.
There is a natural quotient map from $A$ to $\gr_\nu(A)$ given by sending $f\in A$ to $\mathcal F_{\nu,\nu(a)}/\mathcal F_{\nu,>\nu(a)}$, denote its image by $\hat f\in \gr_\nu(A)$.
If the quotients $\mathcal F_{\nu,\gamma}/\mathcal F_{\nu,>\gamma}$ are at most one dimensional, then we say $\nu$ has {\bf one dimensional leaves}.
This property is desirable as it gives an identification
\[
\gr_\nu(A)\to k[S(A,\nu)], \quad \text{by} \quad \hat f_\gamma \mapsto \nu(f_\gamma),
\]
where $\hat f_\gamma\in \mathcal F_{\nu,\gamma}/\mathcal F_{\nu,>\gamma}$ and $f_\gamma\in A$ lies in the preimage of $\hat f_\gamma$ under the quotient map {$\hat{}:A\to \gr_\nu(A)$}.
It is a consequence of Abhyankar's inequality that full rank valuations have one dimensional leaves.
A vector space basis $\mathbb B$ of $A$ is called {\bf adapted to $\nu$}, if $\mathbb B\cap \mathcal F_{\nu,\gamma}$ is a vector space basis for the subvector space $\mathcal F_{\nu,\gamma}$ for all $\gamma\in \Gamma$.
The {\bf Newton--Okounkov body} of the valuation is defined as
\[
\Delta(A,\nu):=\text{conv}\left(\bigcup_{j>0} \{\nu(f)/j:f\in A_j\}\right)
\]
where $A=\bigoplus_{j\in \mathbb Z_{\ge0}}A_j$ is  a $\mathbb Z$-grading on $A$ that is refined by the given multigrading. If $S(A,\nu)$ is a finitely generated semigroup then $\Delta(A,\nu)$ is a rational polytope.

An important definition is the notion of a {\bf Khovanskii basis}: a subset $\mathcal B$ of $A$ whose image in $\gr_\nu(A)$ is an algebra generating set. 
It is not hard to see that if $\mathcal B$ is a Khovanskii basis for $\nu$ then the set $\{\nu(b):b\in\mathcal B\}$ generates the value semigroup \cite[Lemma 2.10]{KM16}. 
If there exists a finite Khovanskii basis for $\nu$ and $A$ then $\nu$ is called {\bf Khovanskii-finite}.

A valuation is called {\bf homogeneous} if it respects the grading on $A$, more precisely if $f\in A$ has homogeneous presentation $\sum_{i} f_i$ then $\nu(f)=\max\{\nu(f_i)\}$.
A valuation is {\bf fully homogeneous} if $\nu(f)=(\deg(f),\nu'(f))$, that is $S(A,\nu)\subset \mathbb Z_{\ge 0}^m\times \Gamma'$.
Any homogeneous valuation is obtained from a fully homogeneous one by composing with an isomorphism of semigroups \cite[Remark 2.6]{IW20}. 
So when studying homogeneous valuation we may without loss of generality assume they are fully homogeneous.
The Newton--Okounkov body of a fully homogeneous valuation can be computed as the intersection of the closure of the cone over $S(A,\nu)$ and the hyperplane indicating degree one in $\mathbb Z^m_{\ge 0}$.

\subsection{Valuations from tropicalization}\label{sec:val from trop}
Before explaining how valuations can be obtained from cones in the tropicalization of an ideal we need a bit more background on a slight generalization of initial ideals: a \emph{higher dimensional analogue of Gr\"obner theory} (see e.g. \cite{FR_Hahn_higherRank}).

Let $S=k[x_1,\dots,x_N]$ and consider as before $f=\sum c_ax^a\in S$. 
We call a matrix $M\in \mathbb Q^{d\times N}$ a {\bf weighting matrix} and together with a linear order $\prec$ on $\mathbb Z^d$ we define the {\bf initial form} of $f$ with respect to $M$ as
\[
\init_M(f):=\init_{M,\prec}(f):=\sum_{b: \ Mb=\max_{\prec}\{Ma:c_a\not =0\}} c_bx^b.
\]
As before we define the {\bf initial ideal} of an ideal $J\subset S$ {\bf with respect to $M$ (and $\prec$)} as $\init_{M,\prec}(J):=(\init_M(f):f\in J)$.
To simplify notation we drop the linear order from the index and simply assume that we have fixed it once and for all.
The Gr\"obner region also has a higher dimensional analogue: the {\bf $d^{\text{th}}$ Gr\"obner region} is denoted $\GR^d(J)$ and defined as the set of all weighting matrices $M\in \mathbb Q^{d\times N}$ such that there exists a term order $<$ on $S$ with $\init_<(J)=\init_<(\init_M(J))$.
Given $<$ let $C^d_<\subset \GR^d(J)$ be the set of all $M$ satisfying the previous relation.
We may also define equivalence classes of weighting matrices by setting $C_M:=\{M'\in \GR^d(J):\init_M(J)=\init_{M'}(J)\}$.
In the higher dimensional case several features of Gr\"obner theory still hold, among these the existence of standard monomial bases.
For example, $\GR^d(J)$ always contains the positive orthant $\mathbb Q_{\ge 0}^{d\times N}$ and if $J$ is homogeneous we have $\GR^d(I)=\mathbb Q^{d\times N}$ (see \cite[Lemma 8.7]{KM16} but be aware that the authors are using the minimum convention which introduces a sign).

We may use weighting matrices to define quasivaluation as follows.
Consider the quotient map $\pi:S\to S/J=A$ and denote by $\bar f$ the coset of $f$ in the quotient. 
For $f=\sum c_a x^a\in S$ set $\tilde \nu_M(f):=\max_\prec\{Ma:c_a\not =0\}$.
This defines a valuation $\tilde \nu_M:S\setminus\{0\}\to \mathbb Z^d$.
By \cite[Lemma 3.2]{KM16} there exists a quasivaluation $\nu_M:A\setminus\{0\}\to (\mathbb Z^d,\prec)$ given for $\bar f\in A$ by
\[
\nu_M(\bar f)=\min_\prec\{\tilde \nu_M(f):f\in \bar f\}
\]
called the {\bf quasivaluation with weighting matrix $M$}. Its associated graded algebra, denoted $\gr_M(A)$, satisfies $\gr_M(A)\cong S/\init_M(J)$.
In particular, this isomorphism gives us standard monomial bases for $\gr_M(A)$: let $<$ be a term order with $M\in C^d_<$ then $\mathbb B_<$ is a vector space basis for $\gr_M(A)$. 
Moreover, we may use $\mathbb B_<$ to compute the values of $\nu_M$: for $\bar f\in A$ let $\bar f=\sum_{\bar x^b\in \mathbb B_<} c_b\bar x^n$ be its expression in $\mathbb B_<$. Then
\[
\nu_M(\pi( f))=\max_{\prec}\{Mb:c_b\not =0\}.
\]

The main aim of Kaveh and Manon in \cite{KM16} is to establish a connection between the toric degenerations from prime cones in a tropicalization to toric degenerations obtained from Newton--Okounkov polytopes.
It relies on the quasivaluations with weighting matrices introduced above.
Suppose there exists a maximal prime cone $\tau\in\trop(J)$ and choose a basis $r_1,\dots,r_d\in \mathbb Q^N$ for the real vector space spanned by $\tau$. 
The quotient $\tau/\mathcal L_J$ is a strongly convex cone (see e.g. \cite[Lemma 2.13]{BMN}) so we may take a maximal linearly independent set of cosets of primitive ray generators of $\tau/\mathcal L_J$. 
Together with a basis of the lineality space this will be our choice for ${\bf r}:=\{r_1,\dots,r_d\}$.
In particular, we set $r_i=\ell_i$ for $1\le i\le m$, see Lemma~\ref{lem:grading and lineality}.
Define $M_{\bf r}:=(r_{ij})_{1\le i\le d,1\le j\le N}$,
where $r_{ij}$ is the $j^{\text{th}}$ entry in $r_i$, so the $r_i$ are the rows of $M_{\bf r}$.

\begin{Proposition}[Proposition 4.2 and 4.6 in \cite{KM16}]
If $\tau$ is a maximal prime cone in $\trop(J)$ then quasivaluation with weighting matrix $M_{\bf r}$ is in fact a full rank valuation with one dimensional leaves.
Its value semigroup is generated by the images of $\bar x_1,\dots,\bar x_N$ and it only depends on $C$.
\end{Proposition}

\subsection{Tropicalization from valuations}\label{sec:trop from val}

Fix a full rank Khovanskii-finite valuation $\nu:A\setminus\{0\}\to \mathbb Z^d$ on a positively graded algebra and domain $A$.
We may assume without loss of generality that $\dim_{\text{Krull}}(A)=d$ (if this was not the case we may apply \cite[Proposition 2.17(e)]{BrunsGubeladze} as the image of $\nu$ is in fact a monoid whose only unit is $0$).
Moreover, we may assume that the underlying total order on $\mathbb Z^d$ is the lexicographic order (if this was not the case we may follow Mora and Robbiano \cite{MoraRobbiano} and represent the order by a $d\times s$ matrix $M$ such that our order coincides with the lexicographic order on $M\mathbb Z^d$).
Choose a finite generating set $a_1,\dots,a_N$ for the value semigroup $S(A,\nu)$ and let $M_\nu$ be $d\times N$ matrix whose columns are $a_1,\dots,a_N$.
Notice that
\[
\text{rank}(M_\nu)=\dim(\text{im}(M_\nu)=\dim(\text{span}_{\mathbb Z}(a_1,\dots,a_n))=\dim(\text{cone}(a_1,\dots,a_n))=\text{rank}(\nu).
\]
In particular, $M_\nu$ is of full rank. 

\begin{Lemma}
Given the generators $a_1,\dots,a_N$ of the value semigroup $S(A,\nu)$ choose $b_1,\dots,b_N\in A$ with $\nu(b_i)=a_i$.
Then the set $\{b_1,\dots,b_N\}$ is a Khovanskii basis. 
\end{Lemma}
\begin{proof}
As $k[S(A,\nu)]\cong \gr_v(A)$ the elements $a_1,\dots,a_N$ form a set of algebra generators for $\gr_v(A)$. 
\end{proof}

Notice further that for dimension reasons the Khovsankii basis $\mathcal B=\{b_1,\dots,b_N\}$ is a generating set for $A$ as $\nu$ is full rank. Hence, we may use it to obtain a presentation of $A$. Define
\[
\pi_{\mathcal B}:S:=k[x_1,\dots,x_N]\to A, \quad \text{by} \quad x_i\mapsto b_i.
\]
Notice that $b_1,\dots,b_N$ not necessarily are of degree one in $A$. 
To have a \emph{graded} presentation of $A$ we endow the polynomial ring  $S$ with a (not necessarily standard) grading given by $\deg(x_i):=\deg(b_i)$.
Then $J_{\mathcal B}:=\ker(\pi_{\mathcal B})$ is a homogeneous ideal and
$S/J_{\mathcal B}\cong A$.
Our valuation $\nu$ is intimately related with the tropicalization of the ideal $J_{\mathcal B}$. 

\begin{Theorem}[\cite{B-quasival}]\label{thm:val and trop}
Let $\nu:A\setminus \{0\}\to \mathbb Z^{d}$ be a full rank valuation with finitely generated value semigroup and let $S/J_{\mathcal B}\cong A$ be the presentation induced by a Khovanskii basis $\mathcal B=\{b_1,\dots,b_n\}$.
Then there exists a maximal prime cone $\tau\in \trop(J_{\mathcal B})$ such that $\nu=\nu_\tau$ and 
\[
k[S(A,\nu)] \cong \gr_\nu(A)\cong S/\init_\tau(J_{\mathcal B}).
\]
In particular, the toric variety associated with $\Delta(A,\nu)$ is the normalization of the central fibre of the Gröbner degeneration $\proj(S/\init_\tau(J_{\mathcal B}))$.
\end{Theorem}

\begin{proof}
Notice that $M:=M_\nu$ the weighting matrix of $\nu$ is of full rank $d \le N$ as $\nu$ is of full rank.
Then by \cite[Theorem 1]{B-quasival} the initial ideal $\init_M(J_{\mathcal B})$ is prime as the value semigroup $S(A,\nu)$ is generated by $\nu(b_1),\dots,\nu(b_N)$. 
Here we use the total order on  $\mathbb Z^{d}$ given by $\nu$.
We may restrict our attention to the case of usual initial ideals as by \cite[Lemma 3]{B-quasival} there exists a weight vector $w\in \mathbb Z^{N}$ such that $\init_w(J_{\mathcal B})=\init_M(J_{\mathcal B})$.
It is left to show that
\begin{itemize}
  \setlength{\itemsep}{1pt}
  \setlength{\parskip}{0pt}
  \setlength{\parsep}{0pt}
    \item $w \in \trop(J_{\mathcal B})$;
    \item $k[S(A,\nu)] \cong S/\init_\tau(J_{\mathcal B})$, where $w \in \tau^\circ$.  
\end{itemize}
The first item follows from \cite[Corollary 3]{B-quasival}.
For the second consider the quasivaluation $\nu_M$.
By Proposition 2.1 $\nu_M$ is a valuation and by \cite[Proposition 1]{B-quasival} it satisfies $\nu=\nu_M$. 
Further, by \cite[Equation 3.3]{B-quasival} we have $S/\init_w(J_{\mathcal B}) \cong \gr_{\nu}(A) \cong k[S(A,\nu)]$.
\end{proof}

\section{Totally positive cluster varieties}\label{sec:cluster}

We assume the reader is familiar with Fomin--Zelevisnky's cluster algebras \cite{FZ02,FZ03,FZ07}. 
In particular, the (positive) Laurent phenomenen \cite{FZ02,GHKK14} and the notion of $\gv$-vectors.
We review the basic notions of cluster varieties and their tropicalization to set up our notation.
As we use both algebraic and geometric language for cluster algebras and varieties, the fact that Fomin--Zelevisnky's cluster variables are local coordinates of the corresponding $\mathcal A$-variety is important.

\subsection{Cluster algebras and varieties}\label{sec:ca and cv}

In this section we assume $k=\mathbb R$ as this is the case we need for the proof of Theorem~\ref{thm:main}.
For $n\in \mathbb Z$ let $[n]_+:=\max\{0,n\}$ and $[n]_-:=\min\{0,n\}$.

\begin{Definition}\label{def:fixed_data}
The {\bf fixed data} $\Gamma$ refers to the following information
\begin{itemize}
  \setlength{\itemsep}{1pt}
  \setlength{\parskip}{0pt}
  \setlength{\parsep}{0pt}
    \item a lattice $N$ with a skew-symmetric bilinear form $\{\cdot,\cdot\}: N\times N\to \mathbb Q$;
    \item a saturated sublattice $N_{\rm mut}\subseteq N$;
    \item two sets of indices $I_{\rm mut}\subset I$ that satisfy $|I|={\rm rank}(N)=:n+m$ and $|I_{\rm mut}|={\rm rank}(N_{\rm mut})=:n$; 
    \item positive integers $d_i$ for $i\in I$ with greatest common divisor 1;
    \item a sublattice $N^\circ\subseteq N$ of finite index such that $\{N_{\rm mut},N^\circ\}\subseteq \mathbb Z$ and $\{N,N_{\rm mut}\cap N^\circ\}\subseteq \mathbb Z$;
    \item the dual lattices  $M:={\rm Hom}_{\mathbb Z}(N,\mathbb Z)$ and $M^\circ:={\rm Hom}_{\mathbb Z}(N^\circ,\mathbb Z)$. 
    \end{itemize}
\end{Definition}
Given $\Gamma$ {\bf seed data} is a tuple $s=(e_i:i\in I)$ such that $\{e_i:i\in I\}$ is a basis of $N$, $\{e_i:i\in I_{\rm mut}\}$ is a basis of $N_{\rm mut}$ and $\{d_ie_i:i\in I\}$ is a basis of $N^\circ$. 
We define the $n\times n$-rectangular integer matrix whose entries are $\epsilon_{{\rm mut};s}:=(\{e_i,e_jd_j\})_{i,j\in I_{\rm mut}}$ and $\epsilon_{{\rm f};s}:=(\{e_i,e_jd_j\})_{i\in I_{\rm mut}, j\not\in I_{\rm mut}}$.
The transpose of $\epsilon_s:=(\epsilon_{\rm mut} \ \epsilon_{\rm f})\in \mathbb Z^{n\times(n+m)}$ is called an {\bf extended exchange matrix} and we denote it by $\tilde B_s=(b_{ij}^s)_{i\in I,j\in I_{\rm mut}}$.
A {\bf fully extended exchange matrix} is an $(n+m)$-square matrix of form
\begin{equation}\label{eq:Mp*}
\widetilde{B_s}=
\left(\begin{matrix}
\epsilon^{T}_{{\rm{uf}};s} & A \\
\epsilon^{T}_{{\rm f};s} & \ast
\end{matrix}
\right)
\end{equation}
where  $A=-
\text{diag}(d_1,\dots,d_{|I_{\rm mut}|})\ 
\epsilon_{{\rm f};s}\ 
\text{diag}(d_{|I_{\rm mut}|+1},\dots,d_{|I|})^{-1}$, and $\ast$ is an integral $(m\times m)$-matrix. 

\begin{Definition}\label{def:mut seed data}
Given the seed data $s$ and an index $k\in I_{\rm mut}$ we define the {\bf mutation in direction $k$ of $s$} to be the seed data $\mu_k(s)=(e_i':i\in I)$ where
\begin{equation}\label{eq:trop mut N}
e_i':=\left\{\begin{matrix}
e_i+[\epsilon_{ik}]_+e_k & i\not =k\\
-e_k & i=k
\end{matrix}\right.
\end{equation}
If $s'$ is obtained from $s$ by a sequence of mutations we write $s'\sim s$.
\end{Definition}

Let $\{f_{i}:i\in I\}$ be the basis of $M^\circ$ dual to the basis $\{e_id_i:i\in I\}$ of $N^\circ$ with respect to $s$ and let $\{f'_i:i\in I\}$ be the dual basis with respect to $\mu_k(s)$. Then
\begin{equation}\label{eq:trop mut M}
    f_i'=\left\{\begin{matrix} -f_k+\sum_{j\in I, j\not =k} [-\epsilon_{kj}]_+f_j & k=i \\ f_i & k\not =i\end{matrix}\right.
\end{equation}

We associate a real algebraic tori to the lattices $N^\circ$ and $M$, that is $T_{N^\circ}(\mathbb R)=N^\circ\otimes_{\mathbb Z}\mathbb R=\spec(\mathbb R[M^\circ])$ and similarly for $T_M(\mathbb R)$.
Given seed data $s$ we obtain coordinates on the tori:
\begin{eqnarray*}
    T_{N^\circ;s}&:=&T_{N^\circ;s}(\mathbb R)=\spec\left(\mathbb R[z^{\pm f_{i;s}}:i\in I]\right)\\
    T_{M;s}&:=&T_{M;s}(\mathbb R)=\spec\left(\mathbb R[z^{\pm e_{i;s}}:i\in I]\right).
\end{eqnarray*}
For every index $k\in I_{\rm mut}$ and every seed data $s$ we define birational maps $\mu_{k;\mathcal A}:T_{N^\circ;s}\dashrightarrow T_{N^\circ;\mu_k(s)}$ and
$\mu_{k;\mathcal X}:T_{M;s}\dashrightarrow T_{M;\mu_k(s)}$
by their pullbacks
\begin{eqnarray*}
    \mu_{k;\mathcal A}^*:\mathbb R(M^\circ) &\to& \mathbb R(M^\circ), \quad z^q\mapsto z^q(1+z^{v_k})^{-\langle e_k,q\rangle},\\
    \mu_{k;\mathcal X}^*:\mathbb R(N) &\to& \mathbb R(N), \quad z^p\mapsto z^p(1+z^{e_k})^{-\{ n,d_ke_k\}},
\end{eqnarray*}
where $q\in M^\circ, p\in N$ and $v_k:=\{e_k,\cdot\}\in M^\circ$ for all $k\in I_{\rm mut}$. 
More generally we write $A_{i;s}:=z^{f_{i;s}}$ where $\{f_{i;s}:i\in I\}$ is the basis of $M^\circ$ associated with the seed data $s=(e_i:i\in I)$.
In these coordinate systems the birational map $\mu_{k;\mathcal A}$ sends $(A_{1;s},\dots,A_{n+m;s})$ to $(A_{1;\mu_k(s)},\dots,A_{n+m;\mu_k(s)})$ (which reveals Fomin--Zelevinsky's original definition of mutation for cluster variables) given by $A_{j;\mu_k(s)}=A_{j;s}$ for $i\not=k$ and
\begin{eqnarray}\label{eq:exchange rel}
A_{k;\mu_k(s)}=A_{k;s}^{-1}\left(\prod_{b^s_{ik>0}}A_{i;s}^{b_{ik}^s}+\prod_{b_{ik}^s<0}A_{i;s}^{-b^s_{ik}}\right).    
\end{eqnarray}
The above relation between the local coordinate systems is called an {\bf exchange relation}. Notice that in particular the coordinates associated with indices in $I\setminus I_{\rm mut}=\{n+1,\dots,n+m\}$ are the same for all seeds. We therefore leave out the index $s$ when referring to these coordinates and simply denote them by $A_{n+1},\dots,A_{n+m}$.

$\mathcal A$- respectively $\mathcal X$-varieties are schemes glued from several copies of the tori $T_{N^\circ;s}$ respectively $T_{M;s}$-- one for every seed-- with transition maps given by $\mu_{k;\mathcal A}$ respectively $\mu_{k;\mathcal X}$.

\begin{Definition}
The {\bf real $\mathcal A$-variety} $\mathcal A_{\Gamma}(\mathbb R)$ associated with the fixed data $\Gamma$ and the seed data $s$ is the scheme glued from several copies of $T_{N^\circ}(\mathbb R)$--one for each $s'\sim s$-- whose transition functions are sequences of the birational mutation maps $\mu_k$.
\end{Definition}

Notice that the transition functions are given by substraction-free expressions in local coordinates. In particular, we may define the cluster variety over a semifield rather than a field, for example over the semifield $\mathbb R_{>0}$.

\begin{Definition}
The {\bf totally positive part $\mathcal A_{\Gamma}(\mathbb R_{>0})\subset \mathcal A_\Gamma(\mathbb R)$} of the real cluster variety $\mathcal A_{\Gamma}(\mathbb R)$ is the set obtained by gluing the totally positive part of the torus
\[
N^\circ\otimes_{\mathbb Z} \mathbb R_{>0}\subset T_N(\mathbb R)
\]
along the transition maps defined by sequences of mutations $\mu_k$.
\end{Definition}

The coordinates used to define the positivity conditions on a cluster variety are all local coordinates of the tori associated with seeds. These coordinates are called the {\bf cluster variables} and they are the generators of the {\bf cluster algebra (of geometric type)}:
\[
A(\Gamma):=\mathbb R[A_{n+1},\dots,A_{n+m}][\mathcal V]\subset \mathcal F,
\]
where $A_{n+1},\dots,A_{n+m}$ are the {\bf frozen cluster variables}, $\mathcal V$ is the set of all mutable cluster variables, i.e. $\mathcal V=\{A_{i;s}:i\in I_{\rm mut},s \text{ seed}\}$, and $\mathcal F$ is the ambient field of rational functions $\mathbb R(M^\circ)=\mathbb R(A_{i;s}:i\in I)$ for any seed $s$. 
A fundamental result about cluster algebras is the following.
\begin{Theorem}[Laurent phenomenon \cite{FZ03},\cite{GHKK14}]
Let $A_{i;s'}$ be a cluster variable in a cluster algebras of geometric type $A(\Gamma)$ and $s$ a seed. 
Then $A_{i;s'}\in \mathbb Z_{\ge 0}[A_{1;s}^{\pm 1},\dots,A_{n;s}^{\pm 1},A_{n+1},\dots,A_{n+m}]$.
\end{Theorem}

\subsection{Principal coefficients}\label{sec:prin}
Given $\Gamma$ we construct another cluster variety from the following doubled data:

\begin{Definition}
Let $\Gamma$ be fixed data. The associated {\bf fixed data with principal coefficients} $\Gamma_{\prin}$ is defined by
\begin{enumerate}
  \setlength{\itemsep}{1pt}
  \setlength{\parskip}{0pt}
  \setlength{\parsep}{0pt}
    \item the lattice $\widetilde{N}:=N\oplus M^\circ$ with skew-symmetric bilinear form  $\{(n_1,m_1),(n_2,m_2)\} := \{n_1,n_2\} + \langle n_1,m_2\rangle - \langle n_2,m_1\rangle$;
    \item the sublattice $\widetilde{N}_{\rm uf}:=N_{\rm mut}\oplus \{0\}\subset\widetilde{N}$;
    \item the index sets $\widetilde{I}_{\rm uf}\subset \widetilde{I}=I\cup I'$, where $I'$ another disjoint copy of $I$ and $\widetilde{I_{\rm mut}}=I_{\rm mut}\subset I$;
    \item the sublattice $\widetilde{N}^\circ$ is $N^\circ\oplus M$
    \item the dual lattices are $\widetilde{M}:={\rm Hom}_{\mathbb Z}(\widetilde{N},\mathbb Z)=M\oplus N^\circ$ and $\widetilde{M}^\circ:={\rm Hom}_{\mathbb Z}(\widetilde{N}^\circ,\mathbb Z)=M^\circ \oplus N$.
\end{enumerate}
\end{Definition}
Given an initial seed $s=(e_i:i\in I)$ for $\Gamma$, we define a seed for $\Gamma_{\prin}$ by $\tilde s:=(\tilde e_i:i\in \widetilde{I})$, where $\tilde e_i:=\{(e_i,0) ,(0,f_j) : i\in I,j\in I' \}$.
The associated basis $\{\tilde f_i:i\in \widetilde{I}\}$ of $\widetilde{M}^\circ$ is given by $\tilde f_i=\{(f_i,0) ,(0,e_j) :i\in I, j\in I'\}$.

\begin{Definition}
The {\bf real $\mathcal A$-variety with principal coefficients} $\mathcal A_{\Gamma^{\prin}_{s}}(\mathbb R)$ is the real $\mathcal A$-variety associated with the fixed data $\Gamma_{\prin}$ and initial seed $\tilde s$.
\end{Definition}

As $\widetilde{I}_{\rm uf}=I_{\rm mut}$ mutations with respect to any index in $I'$ are not permitted. 
In particular, given $\tilde s$ the variables $z^{\tilde f_i}=z^{(0,e_i)}$ for $i\in I'$ may vanish as they never get inverted.
We partially compactify $\mathcal A_{\Gamma_{\prin}^s}(\mathbb R)\subset \bar{\mathcal A}_{\Gamma_{\prin}^s}(\mathbb R)$ 
by allowing the variables $\tilde X_i:=z^{(0,e_i)}$ to vanish, 
which corresponds to gluing patches of form $T_{N^\circ}(\mathbb R)\times \mathbb A^{n+m}_{\mathbb R}$, 
where $\mathbb A^{n+m}_{\mathbb R}$ has coordinates $X_1,\dots,X_{n+m}$ with $n+m=|I|$. 

\begin{Proposition}[Corollary 5.3 and Remark B.10 in \cite{GHKK14}]
\label{prop:degen Aprin to torus}
The map $\pi:N^\circ\oplus M\to M$ given by $(n,m)\mapsto m$ naturally extends to $\pi:\mathcal A_{\Gamma_{\prin}^s}\to T_M$ and compactifies to a toric degeneration
\[
\pi: \bar{\mathcal A}_{\Gamma_{\prin}^s}(\mathbb R)\to \mathbb A^{n+m}_{\mathbb R},
\]
where $X_i=z^{e_i}$ pulls back to $\tilde X_i=z^{(0,e_i)}$. The fibre over the origin satisfies $\pi^{-1}(0)=T_{N^\circ}(\mathbb R)$ while the generic fibre over $(1,\dots,1)\in \mathbb A^{n+m}_{\mathbb R}$ satisfies $\pi^{-1}(1,\dots,1)=\mathcal A_\Gamma(\mathbb R)$.
\end{Proposition}

By construction the positive part $\mathcal A_\Gamma(\mathbb R_{>0})$ can be traced through the above degeneration: we define $\bar{\mathcal A}_{\Gamma_{\prin}^s}(\mathbb R_{>0})$ as the cluster variety glued from patches $T_{N^\circ}(\mathbb R_{>0})\times \mathbb A^{n+m}_{\mathbb R}$ which naturally is embedded in $\bar{\mathcal A}_{\Gamma_{\prin}^s}(\mathbb R)$.
Restricting $\pi$ yields
\begin{eqnarray}\label{eq:prin coef pos}
    \pi_{+}: \bar{\mathcal A}_{\Gamma_{\prin}^s}(\mathbb R_{>0}) \to \mathbb A^{n+m}_{\mathbb R}
\end{eqnarray}
with $\pi^{-1}_+(1,\dots,1)=\mathcal A_\Gamma(\mathbb R_{>0})$ and $\pi^{-1}_+(0)=T_{N^\circ}(\mathbb R_{>0})$.
In special cases the degeneration in Proposition~\ref{prop:degen Aprin to torus} can be further compactified so that a compactification of $\mathcal A_\Gamma(\mathbb R)$ degenerates to a compactification of $T_{N^\circ}(\mathbb R)$, i.e. a toric variety. 
For the general construction we refer to \cite{CMN} and \cite[\S8]{GHKK14}.
By Proposition~\ref{prop:degen Aprin to torus} and the restriction map $\pi_+$ also the totally positive parts of the compactifications of $T_{N^\circ}(\mathbb R)$ are non-empty as they all contain $T_{N^\circ}(\mathbb R_{>0})$.

\subsection{Cluster ensemble maps}\label{sec:p*}

Recall that a cluster ensemble map $p:\mathcal A_{\Gamma} \to \mathcal X_{\Gamma}$ is induced by a lattice map, called a {\bf cluster ensemble lattice map}, $p^*:N\to M^\circ$ satisfying
\begin{itemize}
  \setlength{\itemsep}{1pt}
  \setlength{\parskip}{0pt}
  \setlength{\parsep}{0pt}
    \item $\left.p^*\right|_{N_{\rm mut}}: n \mapsto (\{n,{\cdot}\} :N^\circ \to \mathbb Z)$, 
    \item for $\pi:M^\circ \to M^\circ/N_{\rm mut}^\perp$ the projection, then $\pi \circ p^*:n \mapsto \{n,{\cdot}\}|_{N_{\rm mut}^\circ}$. 
\end{itemize}
Fix a seed $s$ for $\Gamma$ and consider the associated tori $T_{M;s}\subset \mathcal X_{\Gamma}$ and $T_{N^\circ;s}\subset \mathcal A_{\Gamma}$.
Then a cluster ensemble lattice map restricted to the seed tori $p\vert_{s}:T_{N^\circ;s}\to T_{M;s}$ is a monomial map expressing the local coordinates on $\mathcal A_\Gamma$ given by the cluster variables of $s$ in terms of local coordinates of $\mathcal X_\Gamma$ (extending Fomin--Zelevinsky's notion of $\hat y$ variables). 
More precisely, given a seed $s$ a cluster ensemble lattice map $p^*:N\to  M^\circ$ is determined by a fully extended exchange matrix. 
We have identifications $k[N]=k[T_{M;s}]\equiv k[X_{1;s}^{\pm 1},\dots,X_{n+m;s}^{\pm 1}]$ and $k[M^\circ]=k[T_{N^\circ;s}]\equiv k[A_{1;s}^{\pm 1},\dots,A_{n+m;s}^{\pm 1}]$ and so the cluster ensemble lattice map induces
\[
p^*:k[N]\to k[M^\circ], \quad X_{i;s}\mapsto \prod_{j=1}^{n+m} A_{j;s}^{b_{ji}}.
\]
Although for our use we only require $p^*$ to be of full rank in general it is desirable to have an invertible cluster ensemble map, so $p^*$ should be given by a lattice isomorphism.
In this case, combining with the positive Laurent phenomenon (see Proposition~\ref{prop:LP Cor6.3} for a summary) we get a positive parametrization of the cluster variety $\mathcal A_\Gamma$ and its compactification $Y$. 
Moreover, the $\mathcal A_\Gamma$ and $\mathcal X_\Gamma$ cluster varieties are isomorphic in this case.

\begin{Example}\label{exp:A2+frozen}
Consider the cluster algebra of type $\mathtt A_2$ with one frozen variable with skew-symmetric bilinear form is given in the initial seed $s$ by the quiver $1\to 2 \leftarrow \small{\boxed{3}}$.
We choose a fully extended exchange matrix
\[
\widetilde{B_s}=\left(\begin{matrix}
0&1&0\\-1&0&-1\\0&1&1
\end{matrix}\right)
\]
Let $A_1,A_2,A_3$ and $X_1,X_2,X_3$ denote the cluster variables of $s$. Then $p^*$ maps
$X_1\mapsto A_2^{-1}, X_2\mapsto A_1A_3$ and  $X_3\mapsto A_2^{-1}A_3$.
Its inverse is given by $A_1\mapsto X_1X_2X_3^{-1},A_2\mapsto X_1^{-1}$ and $A_3\mapsto X_1^{-1}X_3$.
\end{Example}

It is not hard to verify that for any cluster ensemble map the diagram  on the left hand side below commutes for all $k\in I_{\rm mut}$
\begin{eqnarray}\label{eq:p trop diag}
    \xymatrix{
    \mathcal A_\Gamma \ar[d]_{\mu_{k;\mathcal A}} \ar[r]^{p} & \mathcal X_\Gamma \ar[d]^{\mu_{k;\mathcal X}} \\
    \mathcal A_\Gamma \ar[r]_{p}& \mathcal X_\Gamma
    }\quad \quad \quad 
    \xymatrix{
    \mathcal A_\Gamma(\mathbb Z^T) \ar[d]_{\mu_{k;\mathcal A}^{\trop}} \ar[r]^{p^{\trop}} & \mathcal X_\Gamma(\mathbb Z^T) \ar[d]^{\mu_{k;\mathcal X}^{\trop} } \\
    \mathcal A_\Gamma(\mathbb Z^T) \ar[r]_{p^{\trop}}& \mathcal X_\Gamma(\mathbb Z^T)
    }\quad \quad \quad 
    \xymatrix{
    N_s^\circ \ar[d]_{\mu_{k;\mathcal A}^{\trop}} \ar[r]^{p^{\trop}} & M_s \ar[d]^{\mu_{k;\mathcal X}^{\trop} } \\
    N_{s'}^\circ \ar[r]_{p^{\trop}}& M_{s'}
    }
\end{eqnarray}
The middle diagram in \eqref{eq:p trop diag} is obtained from the left one by passing to the semifield $\mathbb Z^T$ also known as the {\bf Fock--Goncharov tropicaliaztion} (see e.g. \cite[\S2]{GHKK14}). 
Given a seed $s$ the {\bf integer tropical points} of a cluster variety, in this case $\mathcal A_\Gamma(\mathbb Z^T)$ or $\mathcal X_\Gamma(\mathbb Z^T)$, can be identified non-canonically with the cocharacter lattice of the corresponding torus chart.
Fixing another seed $s'$ obtained from $s$ by mutation in direction $k$ we obtain the diagram on the right hand side of \eqref{eq:p trop diag} consisting of lattices and piecewise linear maps.
We proceed by describing the linear regions of these maps.
The linear regions in $N_s^\circ$ resp. $M_s$ are separated by the hyperplane $H_{p^*(e_k)}:=\{p\in N:\langle p^*(e_k),p\rangle=0\}$ resp. $H_{e_k}:=\{q\in M:\langle e_k,q\rangle=0\}$.
Fix the bases $\{e_{i;s}:i\in I\}$ and  $\{e_{i,s'}:i\in I\}$ on $N$ and the bases $\{f_{i;s}:i\in I\}$ and $\{f_{i;s'}:i\in I\}$ on $M^\circ$. 
Then the rightmost diagram in \eqref{eq:p trop diag} splits into two  diagrams of linear maps
\[
\xymatrix{
H^+_{p^*(e_k)} \ar[r]^{\widetilde{B_s}^T}\ar[d]_{\Mu_{k;\mathcal A}^+} & H^+_{e_k} \ar[d]^{\Mu_{k;\mathcal X}^+} \\
N_{s'}^\circ \ar[r]_{\widetilde{B_{s'}}^T} & M_{s'}
}
\quad \quad \text{and} \quad \quad
\xymatrix{
H^-_{p^*(e_k)} \ar[r]^{\widetilde{B_s}^T}\ar[d]_{\Mu_{k;\mathcal A}^-} & H^-_{e_k} \ar[d]^{\Mu_{k;\mathcal X}^-} \\
N_{s'}^\circ \ar[r]_{\widetilde{B_{s'}}^T} & M_{s'}
}
\]
where $H_{p^*(e_k)}^+:=\{p\in N:\langle p^*(e_k),p\rangle\ge 0\}$ etc.
Moreover, $\widetilde{B_s}$ resp. $\widetilde{B_{s'}}$ are  fully extended exchange matrices defining $p^*:M^\circ \to N$ with respect to the bases given by $s$ resp. $s'$.
Further,
\[
\Mu_{k;\mathcal A}^\pm  =
\left(\begin{smallmatrix}
1 \ 0 &  &\cdots & &0\\
0 & \ddots & & &\\
[\pm b_{1k}]_+& \cdots &-1 & \cdots & [\pm b_{n+m,k}]_+ \\
 & & & \ddots & \ 0 \\
0 & & \cdots &  & 0\ 1
\end{smallmatrix}\right) 
\text{ in } H^\pm_{p^*(e_k)},\quad \text{ resp. }  \quad
\Mu_{k;\mathcal X}^\pm  =
\left(\begin{smallmatrix}
1 & 0\ \cdots & [\mp b_{k1}]_+ & \cdots & 0 \\
0 &\ddots & \vdots & & \vdots \\
 & & -1 & & \\
\vdots & & &\ddots & \vdots \\
0 & \cdots & [\mp b_{k,n+m}]_+ & \cdots & 1
\end{smallmatrix}
\right)  \text{ in } H^\pm_{e_k}.
\]
Notice that 
$\Mu_{k;\mathcal A}^\pm=\left(\Mu_{k;\mathcal A}^\pm\right)^{-1}$ and $\Mu_{k;\mathcal X}^\pm=\left(\Mu_{k;\mathcal X}^\pm\right)^{-1}$. 
Moreover, these maps encode {\it matrix mutation}, {\it i.e.} as in \cite[Equation (2.5)]{FWZ} we have
\begin{eqnarray}\label{eq:matrix mut}
    \widetilde{B_{s'}}^T =  \Mu^\pm_{k;\mathcal X} \widetilde{B_s}^T  \Mu^\pm_{k;\mathcal A}.
\end{eqnarray}

\subsection{Valuations of g-vectors}\label{sec:gv}

Initiated in \cite[Proof of Proposition 4.3]{Labardini_et_al_CC-alg} and further elaborated by Fan Qin  \cite[Definition 3.1.1]{Qin17} we recall the dominance order on the lattice $M^\circ$.

\begin{Definition}
\label{def:dominance order}
Given a seed $s=(e_i:i\in I)$ and a cluster ensemble lattice map $p^*:N\to M^\circ$ we define the relation $\prec_{s}$ on ${M}^{\circ}$ for $m_1,m_2\in M^\circ$ by
\[
m_1 \prec_{s} m_2 \ \Leftrightarrow \ \lambda m_2= \lambda m_1 + {p}^{\ast}_1(n) \text{ for some }n\in {N}^+_{s} \text{ and }\lambda\in \mathbb Z_{>0},
\]
where $N^+_s$ is the set of lattice points in the cone spanned by the elements of $s$.
Since ${p}_1^{\ast}$ is injective we have that $({M}^{\circ}, \prec_{s})$ is a partially ordered group.
We call $\prec_{s} $ {\bf the dominance order associated to $s$}. 
A linear order on ${M}^{\circ}$ refining $\prec_{s}$ will be called a {\bf linear dominance order} and we denote it by $<_{s}$. More precisely, for any linear order $\lessdot$ on $M^\circ$ the {\bf linear refinement of $\prec_s$ by $\lessdot$ 
} is given by
\[
m_1<_s m_2 \ \Leftrightarrow \ m_1\prec_s m_2, \ \text{ or } \ m_1\not\prec_s m_2 \text{ and } m_1\lessdot m_2.
\]
\end{Definition}
Linear refinements of the dominance order exist as the dominance order is normal \cite[Theorem 1]{Fuchs50}, for more details we refer to \cite{BCMN_NO}.

\begin{Lemma}\label{lem:iso ordered}
Let $<_s$ be a linear refinement of $\prec_s$ by a linear order $\lessdot$ that satisfies $0\lessdot m$ for all $m\in p^*(N_s^+)$.
If additionally $p^*:N\to M^\circ$ is of full rank, it extends to an isomorphism of linearly ordered abelian groups $p^*:(N_{\mathbb Q},\lessdot)\to (M^\circ_{\mathbb Q},<_s)$.
\end{Lemma}
\begin{proof}
By a little abuse of notation we denote the linear order on $N_{\mathbb Q}$ induced from $\lessdot$ by pullback along $p^*$ also by $\lessdot$. So $0\lessdot n$ holds for all $n\in N_s^+$.
Let $<$ denote the order on $N_{\mathbb Q}$ induced by pullback of $<_s$ along $p^*$. For $n_1,n_2\in N$ we have
\begin{eqnarray*}
n_1<n_2 \ &\Leftrightarrow& \ p^*(n_1)<_s p^*(n_2) \\
&\Leftrightarrow& \ p^*(n_1)\prec_s p^*(n_2), \text{ or } p^*(n_1)\not\prec_s p^*(n_2) \text{ and } p^*(n_1)\lessdot p^*(n_2) \\
&\Leftrightarrow& \ n_2=n_1+n \text{ for some } n\in (N_s^+)_{\mathbb Q}, \text{ or } \not\exists\ n\in (N_s^+)_{\mathbb Q}: n_2=n_1+n \text{ and } n_1\lessdot n_2
\end{eqnarray*}
However, if $n_2=n_1+n$ for some $n\in (N_s^+)_{\mathbb Q}$ we have $0\lessdot n$ and so $n_1 = n_1+0 \lessdot n_1+n = n_2$. So $n_1 <n_2$ holds if and only if $n_1\lessdot n_2$.
\end{proof}

\begin{Definition}\label{def:Langlands dual}
We define the {\bf Langlands dual data} $\Gamma^\vee$ of $\Gamma$ as
\begin{itemize}
  \setlength{\itemsep}{1pt}
  \setlength{\parskip}{0pt}
  \setlength{\parsep}{0pt}
    \item $N^\vee:=N^\circ$ with saturated subalttice $d\cdot N=:(N^\vee)^\circ$ where $d:=\text{lcm}(d_1,\dots,d_{n+m})$ is the least common multiple;
    \item the skew-symmetric bilinear form is $\{\cdot,\cdot\}^\vee:N^\vee\times N^\vee\to \mathbb Q$, where $\{\cdot,\cdot\}^\vee:=d^{-1}\{\cdot,\cdot\}$;
    \item $I^\vee:=I$ and $I^\vee_{\rm mut}:=I_{\rm mut}$;
    \item for all $i\in I^\vee$ we define $d_i^\vee:=d_i^{-1}d$;
    \item the dual lattices are $M^\vee=M^\circ$ y $(M^\vee)^\circ=d^{-1}M$.
\end{itemize}
Note that if $d_i=1$ for all $i\in I$ then $\Gamma^\vee=\Gamma$.
Given the seed data $s=(e_i:i\in I)$ for $\Gamma$ we define {\bf Langlands dual seed data} for $\Gamma^\vee$ by
$s^\vee:=(e_i^\vee:i\in I)$, where $e_i^\vee:=d_ie_i$.
\end{Definition}    

Given a cluster ensemble lattice map $p^*:N\to M^\circ$ for $\Gamma$, we obtain a natural cluster ensemble lattice map for $\Gamma^\vee$, $(p^\vee)^*:N^\vee\to (M^\vee)^\circ$ called the {\bf Langlands dual cluster ensemble lattice map}. 
In Table~\ref{tab:p maps and matrices} we summarize the lattice maps obtained from a choice of cluster ensemble map and the matrices representing them with respect to a choice of seed $s$ for $\Gamma$.

\begin{table}[]
    \centering
    \begin{tabular}{c|c|c|c}
       lattice map  & basis of domain & basis of codomain & matrix \\\hline
        $p^*:N\to M^\circ$ & $s=\{e_i:i\in I\}$ & $\{f_i:i\in I\}$ & $\widetilde{B_s}$ \\
        $p^{\trop}:N^\circ\to M$ & $\{d_ie_i:i\in I\}$ & $\{d_if_i:i\in I\}$ & $\widetilde{B_s}^T$\\
        $(p^\vee)^*:N^\vee\to (M^\vee)^\circ$ & $s^\vee=\{e_i^\vee:i\in I\}$ & $\{f_i^\vee:i\in I\}$ & $-\widetilde{B_s}^T$ \\ 
        $(p^\vee)^{\trop}:(N^\vee)^\circ\to M^\vee$ & $\{d_i^\vee e_i^\vee:i\in I\}$ & $\{d_i^\vee f_i^\vee:i\in I\}$ & $-\widetilde{B_s}$
    \end{tabular}
    \caption{Lattice maps given by a cluster ensemble map and the matrices representing them with respect to the bases induced by a choice of seed $s$.}
    \label{tab:p maps and matrices}
\end{table}

\medskip

\begin{Theorem}[\cite{GHKK14}]\label{thm:theta basis}
Assuming that the full Fock--Goncharov conjecture holds for $\mathcal A_\Gamma$, the cluster algebra $A(\Gamma)$ has a basis of {\bf $\vartheta$-functions} $\{\vartheta_q:q\in \mathcal X_{\Gamma^\vee}(\mathbb Z^T)\}$ indexed by the integer tropical points of the cluster dual $\mathcal X_{\Gamma^\vee}$. 
In particular, the basis contains all cluster variables and cluster monomials whose indices in $\mathcal X_{\Gamma^\vee}(\mathbb Z^T)$ coincide with Fomin--Zeleviskys {$\gv$-vectors}.
\end{Theorem}

The $\gv$-vectors associated with cluster variables in the cluster algebra $A(\Gamma)$ naturally live in $\mathcal X_{\Gamma^\vee}(\mathbb Z^T)$, which can be identified (non-canonically) with $M^\vee$ for every choice of seed $s^\vee$.
The $\gv$-vectors play a key role for us due to the following result that holds true more generally for certain middle cluster algebras.

\begin{Lemma}[\cite{BCMN_NO}, see also \cite{FO20}]\label{g_valuation}
Let $\gv_s: A(\Gamma) \setminus \{ 0\} \to (M^\vee ,<_s)$ be the function defined by $\gv_s(f):= \min{}_{<_s}\{m_1, \dots , m_t\}$,
where $f=c_1\vartheta_{m_1} + \dots + c_t\vartheta_{m_t}$ with $c_i\not=0$ for all $i=1,\dots,t$ is the expression of $f$ in the theta basis of $A(\Gamma)$. 
Then $\gv_s$ is a valuation with one dimensional leaves and the theta basis $\{ \vartheta_m\in A(\Gamma) : m\in\mathcal X_{\Gamma^\vee}(\mathbb Z^T)\}$ is adapted for $\gv_s$.
\end{Lemma}

More precisely, when $A(\Gamma)$ is the homogeneous coordinate ring of $Y$ (rather then the regular functions on $\mathcal A_\Gamma$) then $\mathcal X_\Gamma(\mathbb Z^T)$ contains a {\it positive set} indexing those theta functions on $\mathcal A_\Gamma$ that lie in $A(\Gamma)$, see \cite[\S8.5]{GHKK14} and \cite{CMN}. 
After identifying $\mathcal X_{\Gamma^\vee}(\mathbb Z^T)$ with $M^\vee$ the positive set consists of the integral points in a polyhedral cone.
For more details on how to compute the cone we refer to forthcoming work \cite{BCMN_NO}.

\medskip
Assume we have a cluster ensemble lattice map $p^*:M^\circ\to N$ that is of full rank. 
Then the Langlands dual map and the tropicalizations of $p$ and $p^\vee$ are also full rank and a $\gv$-vector in $M^\vee$ has a unique preimage in $(N^\vee)^\circ_{\mathbb Q}$.
When we change from one initial seed to another the g-vectors change according to the tropical $\mathcal X$-mutation and by the above their preimages under $(p^\vee)^{\trop}$ change according to the tropical $\mathcal A$-mutation.
Computationally, the tropical $\mathcal A$-mutation is simpler than the the tropical $\mathcal X$-mutation, as in coordinates only a single entry changes.

\begin{Proposition}[Laurent phenomenon and Corollary~6.3 in \cite{FZ07}]\label{prop:LP Cor6.3}
The cluster algebras $A(\Gamma)$ and $A(\Gamma_{\prin}^s)$ are of the same cluster type for all seeds $s$ of $A(\Gamma)$. 
Moreover, any cluster variable $A$ of $A(\Gamma)$ is a Laurent polynomial in the cluster variables $\{\tilde A_{1;s},\dots,\tilde A_{n+m;s}\}$ with positive integer coefficients for every seed $s$.
The Laurent polynomial for $\tilde A$ can be obtained from the Laurent polynomial for $\tilde A\in A(\Gamma_{\prin}^s)$ by evaluating all principal coefficient variables to one.
\end{Proposition}

\subsection{Proof of the main result}\label{sec:main}
Before proving the main result we explain it in an example continuing Example~\ref{exp:A2+frozen}.

\begin{Example}
As the cluster algebra $A(\Gamma)$ defined by $1\to 2 \leftarrow \small{\boxed{3}}$ is of finite type we may choose the set of all cluster variables as a Khovanskii basis $\mathcal B=\{A_1,A_2,A_3,A_4,A_5,A_6\}$.
Here $A_4=\mu_1^*(A_1), A_5=\mu_2^*\mu_1^*(A_2)$ and $A_6=\mu_1^*\mu_2^*\mu_1^*(A_1)$.
We compute the $\gv$-vectors of all cluster variables with respect to the seeds $s$ and $s'=\mu_1(s)$ which are the columns of the matrices
\begin{align*}
G_{s;\mathcal B}= 
\left(\begin{smallmatrix} 
1&0&0&-1&-1&0\\0&1&0&1&0&-1\\0&0&1&0&0&0
\end{smallmatrix}\right)\quad  \text{and} \quad
G_{s';\mathcal B}=
\left(\begin{smallmatrix} 
-1&0&0&1&1&0\\0&1&0&0&-1&-1\\0&0&1&0&0&0
\end{smallmatrix}\right)
\end{align*}
Applying the tropicalized Langlands dual cluster ensemble map reduces to multiplying  $G_{s;\mathcal B}$ respectively $G_{s';\mathcal B}$ by 
\[
-\widetilde{B_s}^{-T}=\left(\begin{smallmatrix} -1&-1&1\\1&0&0\\1&0&-1 \end{smallmatrix}\right) \quad \text{respectively} \quad -\widetilde{B_{s'}}^{-T}=\left(\begin{smallmatrix} -1&1&-1\\-1&0&0\\-1&0&-1 \end{smallmatrix}\right).
\]
Then by Theorem~\ref{thm:main} the rows of the following matrices span adjacent maximal prime cones in $\trop^+(J_{\mathcal B})$:
\[
\tau_s:\left(\begin{smallmatrix} -1&-1&1&0&1&1\\1&0&0&-1&-1&0\\1&0&-1&-1&-1&-1&0 \end{smallmatrix}\right)
\quad \text{and} \quad
\tau_{s'}:\left(\begin{smallmatrix} 1&1&-1&-1&0&-1\\1&0&0&-1&-1&0\\1&0&-1&-1&-1&-1&0
\end{smallmatrix}\right).
\]
To verify the statement we first compute $J_{\mathcal B}$: in this case it is the ideal generated by the exchange relations (in general saturation is necessary)
\[
J_{\mathcal B}=(A_1A_4-1-A_2,A_2A_5-A_3-A_4,A_6A_4-A_3-A_5,A_5A_1-A_6-1,A_6A_2-A_3A_1-1).
\]
The initial ideals corresponding to $\tau_s$ and $\tau_{s'}$ are
\begin{eqnarray*}
\init_{\tau_s}(J_{\mathcal B}) &=& (A_4A_6 - A_5, A_2A_6 - 1, A_1A_5 - 1, A_2A_5 - A_4, A_1A_4 - A_2),\quad \text{and}\\
\init_{\tau_{s'}}(J_{\mathcal B}) &=& (A_2A_6 - 1, A_1A_4 - 1, A_1A_5 - A_6, A_2A_5 - A_4, A_4A_6 - A_5).
\end{eqnarray*}
\end{Example}

The proof Theorem~\ref{thm:main} relies on a Lemma about initial forms of exchange relations with respect to weighting matrices defined by $\gv$-vectors.
For our cluster algebra $A(\Gamma)$ we consider two seeds $s$ and $s'$ that are related by mutation in direction $k$.
Let $\mathcal B=\{A_1,\dots,A_N\}$ be a Khovanskii basis for $\gv_s$ and $\gv_{s'}$ containing all cluster variables of $s=(A_1,\dots,A_{n+m})$ and $s'=s-\{A_k\}\cup \{A_{k}'\}$.
The exchange relation between $s$ and $s'$ is of form
\[
A_kA_k'=\prod_{i\in I}A_i^{[b_{ik}]_+} + \prod_{i\in I}A_i^{-[b_{ik}]_-}
\]
which corresponds to an element $f=x_kx_k'-\prod_{i\in I}x_i^{[b_{ik}]_+} - \prod_{i\in I}x_i^{-[b_{ik}]_-}\in J_{\mathcal B}$.
Here $\tilde B_s=(b_{ij})_{i\in[n+m],j\in [n]}$.
Fix a linear dominance order on $M^\circ$ and consider the matrices $G_{s;\mathcal B}$ and $G_{s';\mathcal B}$ whose columns are $\gv$-vectors of elements in $\mathcal B$ with respect to $\gv_s$ and $\gv_{s'}$.
We consider $G_{s;\mathcal B}$ and $G_{s';\mathcal B}$ as weighting matrices for $k[x_1,\dots,x_N]$ with respect to a linear refinement of the dominance order that we fix once and for all.

\begin{Lemma}\label{lem:ideals s and s'}
We have $\init_{G_{s;\mathcal B}}\left(f\right) = x_kx_k'- \prod_{i\in I}x_i^{-[b_{ik}]_-}$ and $ \init_{G_{s';\mathcal B}}\left(f\right) = x_kx_k'- \prod_{i\in I}x_i^{[b_{ik}]_+}$.
\end{Lemma}

\begin{proof}
The $\gv$-vectors with respect to $s$ are $\gv_s(A_i)=f_i$. 
By \cite[Proposition 6.6]{FZ07} we have
\begin{eqnarray*}
    \gv_s(A_k')=-f_k+\sum_{i\in I}[b_{ik}]_+f_i-\sum_{i\in I}b_{ik}f_i 
    = -f_k-\sum_{i\in I}[b_{ik}]_-f_i.
\end{eqnarray*}
For the monomials in the exchange relation we have
\[
\gv_s\left(\prod_{i\in I}A_i^{[b_{ik}]_+}\right) = \sum_{i\in I} [b_{ik}]_+f_i, \quad \gv_s\left(\prod_{i\in I}A_i^{-[b_{ik}]_-}\right) = -\sum_{i\in I} [b_{ik}]_-f_i
\]
which  implies the claim for $\init_{G_{s;\mathcal B}}(f)$. 
Reversing the roles of $s$ and $s'$ changes the sign of the matrix entries $b_{ik}$ which yields the claim for $\init_{G_{s';\mathcal B}}(f)$.
\end{proof}

\begin{Proposition}\label{cor:max cones}
Let $s,s'$ and $J_{\mathcal B}$ be as above.
Fix a full rank fully extended exchange matrix $\widetilde{B_s}$ for $s$ and let $\widetilde{B_{s'}}$ be the matrix obtained from by matrix mutation \eqref{eq:matrix mut}.
Then the columns of $-\widetilde{B_s}^{-T}G_{s;\mathcal B}$  \text{and} $-\widetilde{B_{s'}}^{-T}G_{s';\mathcal B}$ span two adjacent $(n+m)$-dimensional cones in $\trop(J_{\mathcal B})\subset\mathbb R^{N}$, denote them by $\tau_s$ and $\tau_{s'}$ that satisfy
\[
\init_{\tau_s}(J_{\mathcal B})=\init_{G_{s;\mathcal B}}(J_{\mathcal B}) \quad \text{and} \quad \init_{\tau_{s'}}(J_{\mathcal B})=\init_{G_{s';\mathcal B}}(J_{\mathcal B}).
\]
Moreover, $\tau_s$ and $\tau_{s'}$ are two adjacent maximal prime cones in $\trop(J_{\mathcal B})$.
\end{Proposition}

\begin{proof}
We first prove the polyhedral statement:
the set of cluster variables $\{A_1,\dots,A_N\}$ contains all $A_{i;s}$ and $A_{i;s'}$, so both matrices $G_{s;\mathcal B}$ and $G_{s',\mathcal B}$ contain an $(n+m)$-square identity matrix and are of full rank.
As $\widetilde{B_s}$ is of full rank (and the rank is an invariant under matrix mutation) we deduce that the columns $u_1,\dots,u_{n+m}$ of $-\widetilde{B_s}^{-T}G_{s;\mathcal B}$ and $u'_1,\dots,u'_{n+m}$ of $-\widetilde{B_{s'}}^{-T}G_{s';\mathcal B}$ span two $(n+m)$-dimensional cones in $\mathbb R^N$.
Notice that $-\widetilde{B_s}^{-T}$ is the matrix defining $((p^\vee)^{\trop})^{-1}$ with respect to $s$ (see Table~\ref{tab:p maps and matrices}) so by the right most diagram in \eqref{eq:p trop diag} the matrices
$-\widetilde{B_s}^{-T}G_{s;\mathcal B}$ and $-\widetilde{B_s'}^{-T}G_{s';\mathcal B}$
are related by the tropical $\mathcal A$-mutation. Hence, they differ in only one row, namely the one corresponding to the mutation direction $k$. So, $u_i=u_i'$ unless $i=k$ which implies that the cones share a facet.

For the second part of the statement let $<_{s^\vee}$ be the linear refinement of $\prec_{s^\vee}$ by some linear order $\lessdot$ on $(M^\vee)^\circ\supset M^\vee$.
Notice that $(p^\vee)^{\trop}$ induces an isomorphism between the rational vector spaces $d\cdot N_{\mathbb Q}\to M^\circ_{\mathbb Q}$.
Hence, by Lemma~\ref{lem:iso ordered} $\lessdot$ induces a linear order on $d\cdot N_{\mathbb Q}$, which without loss of generality we may assume to be the standard lexicographic order.
So by \cite[Lemma 8.8]{KM16} we have
\[
\init_{\left((p^\vee)^{\trop}\right)^{-1}(G_{s;\mathcal B}),\lessdot}(J_{\mathcal B}) = \init_{u_{n+m}}(\cdots (\init_{u_1}(J_{\mathcal B}))\cdots) = \init_{u_{1}+\epsilon_2u_2+\dots+\epsilon_{n+m}u_{n+m}}(J_{\mathcal B}), 
\]
where $\epsilon_2,\dots,\epsilon_{n+m}>0$ are chosen carefully according to \cite[Proposition 1.13]{Stu96}.
Moreover, by Lemma~\ref{lem:iso ordered}
\[
\init_{G_{s;\mathcal B},{<_{s^\vee}}}(J_{\mathcal B})= \init_{\left((p^\vee)^{\trop}\right)^{-1}(G_{s;\mathcal B}),\lessdot}(J_{\mathcal B}).
\]
Now by \cite[Theorem 1]{B-quasival} the initial ideal $\init_{G_{s;\mathcal B},{<_{s^\vee}}}(J_{\mathcal B})$ is prime as $\{A_1,\dots,A_N\}$ is a Khovanskii basis for $\gv_{s}$. 
In particular, $u_1,\dots,u_{n+m}\in \trop(J_{\mathcal B})$ and $u_{1}+\epsilon_2u_2+\dots+\epsilon_{n+m}u_{n+m}$ is contained in a maximal prime cone $\tau_s$ of $\trop(J_{\mathcal B})$. 
Notice that $\tau_s$ intersects the linear span $\langle u_1,\dots,u_{n+m}\rangle$ non-trivially as $\init_{\tau_s}(J)=\init_{u_{1}+\epsilon_2u_2+\dots+\epsilon_{n+m}u_{n+m}}(J)$.

Repeating the argument for $s'$ we find a maximal prime cone $\tau_{s'}$ intersecting the linear span of $\langle u_1,\dots,\hat u_k,u_k',\dots,u_{n+m}\rangle$.
Moreover, by Lemma~\ref{lem:ideals s and s'} $\tau_s$ and $\tau_{s'}$ intersect in a common facet.
In particular, $u_k$ generates a ray that lies in $\tau_s \setminus (\tau_s\cap \tau_{s'})$. 
\end{proof}

\begin{proof}[Proof of Theorem~\ref{thm:main}]
Let $s'=\mu_k(s)$ for some $k\in I_{\rm mut}$. 
Given Proposition~\ref{cor:max cones} the only claim left to prove is that the initial ideals $\init_{\tau_s}(J_{\mathcal B})$ and $\init_{\tau_{s'}}(J_{\mathcal B})$ are totally positive. 
To see this we consider the degeneration of the cluster variety given by the positive part of the real $\mathcal A$-variety with principal coefficients \eqref{eq:prin coef pos}. 
First, note that the ring of regular functions on the $\mathcal A$-variety is the \emph{upper cluster algebra}
\[
\up(\mathcal A_\Gamma):=H^0(\mathcal A_\Gamma,\mathcal O_{\mathcal A_\Gamma})=\bigcap_{s\text{ seed}} k[A_{i;s}^{\pm 1}:i\in I]
\]
By the Laurent phenomenon, we have $A(\Gamma)\subset \up(\mathcal A_\Gamma)$. In particular, $\spec(A(\Gamma))$ is a partial compactification of $\mathcal A_{\Gamma}$ and we have $\proj(A(\Gamma))$ is a compactification of $\mathcal A_\Gamma$. 
We further compactify $\bar{\mathcal A}_{\Gamma^{\prin};s}$ by letting the frozen variables vanish \cite[Construction B.9]{GHKK14}, denote this by $\overline{\mathcal A}_{\Gamma^{\prin;s}}^{\fr}$.
In this case \cite[Theorem 8.30]{GHKK14} applies and gives a toric degeneration of $\proj(A(\Gamma))$ (by principal coefficients) whose central fibre is a projective toric variety defined by the Newton--Okounkov polytope of the valuation $\gv_s$.
In particular, the central fiber of this degeneration is the normalization of the central fiber of the Gröbner degeneration associated with $G_{s;\mathcal B}$ by Theorem~\ref{thm:val and trop}. 

As the central fibre of the principal coefficient degeneration contains $\mathbb R_{>0}^{n+m}$ (Proposition~\ref{prop:degen Aprin to torus}), so does the central fibre of the Gröbner degeneration $\proj(k[x_1,\dots,x_N]/\init_{\tau_s}(J_{\mathcal B}))$.
The coordinates of the totally positive torus $\mathbb R_{>0}^{n+m}$ correspond to the variables of the seed $s$.
Consider the lifts of the cluster variables $A_1,\dots,A_N\in A(\Gamma)$ to $\tilde A_1,\dots,\tilde A_N\in A(\Gamma_{\prin}^s)$. 
They are Laurent polynomials with positive coefficients in the variables $\tilde A_{1;s},\dots,\tilde A_{n+m;s}$ that depend on the deformation parameters $X_{1;s},\dots,X_{n+m;s}$.
In the central fibre of $\mathcal A_{\Gamma_{\prin}^s}(\mathbb R)$ the coordinates $\tilde A_1,\dots,\tilde A_N$ are Laurent monomials in $\tilde A_{1;s},\dots,\tilde A_{n+m;s}$.
In particular, if $\tilde A_{1;s},\dots,\tilde A_{n+m;s}$ take positive real values, then so do  $\tilde A_1,\dots,\tilde A_N$.
Hence, $V(\init_{\tau_s}(J_{\mathcal B}))\cap \mathbb R^{n+m}_{\ge 0}\not =\varnothing$ and so by Proposition~\ref{prop:tot pos and vanishing} the ideal $\init_{\tau_s}(J_{\mathcal B})$ is totally positive.
\end{proof}

\begin{Corollary}
With the assumptions as in Theorem~\ref{thm:main}, the columns of $-\widetilde{B_s}^{-T}G_{s;\mathcal B}$ corresponding to frozen directions lie in the lineality space $\mathcal L_{J_{\mathcal B}}$.
\end{Corollary}
\begin{proof}
Recall that by our assumptions $Y$ contains the (partially compactified) cluster variety $\bar{\mathcal A}_\Gamma=\bigcup_{s\text{ seed}} T_{N_{\rm uf;s}^\circ}\times \mathbb A^m$, where the coordinates of the affine space are given by the frozen cluster variables.
In particular, the torus $T_{N^\circ/N_{\rm mut}^\circ}\subset \mathbb A^m$ acts on $\bar{\mathcal A}_\Gamma$ and the action extends to $Y$.
The torus action on the central fibre of the degeneration by principal coefficients extends this action and $\gv$-vectors are the torus weights.
Consider $J_{\mathcal B}$ with $X=V(J_{\mathcal B})\subset\mathbb P^{N-1}$ and let $x_1,\dots,x_N$ be the coordinates of $\mathbb P^{N-1}$.
Then $J_{\mathcal B}$ homogeneous with respect to a $\mathbb Z^m$-grading defined by $\deg(x_i)=(\gv_s(A_i)_j)_{j\in I-I_{\rm mut}}$ (only those entries of the $\gv$-vector corresponding to frozen directions). 
Then $J_{\mathcal B}$ is also homogeneous with respect to the grading given by the preimages under $-\widetilde{B_s}^T$.
The rest follows by \cite[Lemma 2.10]{BMN}.
\end{proof}

\begin{Remark}\label{rmk:p not full rank}
The full rank assumption on the fully extended exchange matrix may be loosened. In some examples a cluster ensemble lattice map admitting a kernel encodes interesting information about torus actions on $\mathcal A$ related to the Picard group of $Y$.
For example, in case of the Grassmannian $\Gr(k,\mathbb C^n)$ the cluster variety $\mathcal A_\Gamma$ is contained in the affine cone $C(\Gr(k,\mathbb C^n))$ rather than $\Gr(k,\mathbb C^n)$ itself.
Considering a quotient by a one dimensional torus $T$ we obtain a cluster variety $\mathcal A_\Gamma/T\subset \Gr(k,\mathbb C^n)$.
If a cluster ensemble lattice map associated to $p:\mathcal A_\Gamma\to \mathcal X_\Gamma$ admits a kernel $K$ that satisfies $T_{K^*}=T$, then $p$ descents to an isomorphism $\bar p:\mathcal A_{\Gamma}/T \to \mathcal X_{\Gamma/T}\subset \mathcal X_\Gamma$.
The cluster varieties $\mathcal A_{\Gamma}/T$ and $\mathcal X_{\Gamma/T}$ can also be obtained from fixed data $\Gamma/T$ obtained from $\Gamma$ by deleting a frozen direction.
Theorem~\ref{thm:main} can be extended to this case.
\end{Remark}

\section{Application I: Grassmannians and plabic graphs}\label{sec:Grass}

We keep this section rather brief and refer to the rich literature that already exists on the subject such as \cite{Akihiro_comb-mut-plabic,B-birat,B-quasival}.
\medskip

In \cite{RW17} Rietsch and Williams construct Newton--Okounkov polytopes for Grassmannians from their cluster structure.
The Plücker coordinates of the Grassmannian are cluster variables \cite{Sco06}.
If a given seed $s$ consists only of Plücker coordinates one can associated a {\it plabic graphs} $\mathcal G_s$ to it \cite{MarshScott}.
The valuation determining the Newton--Okounkov polytope of a plabic graph seed can be computed using the combinatorics of \emph{flows} on the plabic graph \cite[Definition 8.1]{RW17}.
We denote this {\it flow valuation} by $\val_{s}:A_{k,n}\setminus\{0\}\to \mathbb Z^{k(n-k)}$, where $A_{k,n}=A(\Gamma)$ is the homogeneous coordinate ring of $\Gr(k,n)$ with respect to its Plücker embedding and the cluster algebra of interest.
The lattice $\mathbb Z^{k(n-k)}$ can be identified with $M^\vee$ in the Langlands dual fixed data $\Gamma^\vee$.
By \cite{SW_cyclic_sieving} the full Fock--Goncharov conjecture is satisfied in this case and more precisely, the conditions of Theorem~\ref{thm:theta basis} (see also \cite{BCMN_Grass}).

From the combinatorics of plabic graphs it follows directly that when two plabic graph seeds $s$ and $s'$ are related by mutation in direction $k$ then only the $k^{\text{th}}$ entry of $\val_s(p_J)$ and $\val_{s'}(p_J)$ differ for every Plücker coordinate $p_J$.
The previous statement holds more generally for arbitrary seeds.
In the context of \cite{BFFHL} Theorem~\ref{thm:Gr plabic trop} was a conjecture posed during discussions among Xin Fang, Ghislain Fourier,  Milena Hering, Kiumars Kaveh,  Martina Lanini, Christopher Manon and myself.
We are now prepared to prove it.

\begin{proof}[Proof of Theorem~\ref{thm:Gr plabic trop}]
Let $s$ be a seed satisfying the assumption that the Newton--Okounkov polytope of $\val_s$ is integral, so the Plücker coordinates form a Khovanskii basis. 
It has been shown in \cite{BCMN_Grass} that for every $J\in \binom{n}{k}$ the flow valuation $\val_s(p_J)$ and the $\gv$-vector valuation $\gv_s(p_J)$ are related by a fully extended exchange matrix $\widetilde{B_s}$:
\[
-\widetilde{B_s}^T(\val_s(p_J))=\gv_s(p_J)-f_{(n-k)\times k},
\]
where $f_{(n-k)\times k}\in M^\circ$. 
Although $\widetilde{B_s}$ is not of full rank, we can extend Theorem~\ref{thm:main} to this case using Remark~\ref{rmk:p not full rank}.
Then it applies to all pairs of seeds related by mutation whose associated Newton--Okounkov polytopes are integral. 
Hence, the rows of $M_s$ are rays in the (positive part of) tropical Grassmannian $\trop^+(J_{k,n})$.
\end{proof}

\section{Application II: Flag varieties and FFLV degenerations}\label{sec:flag}

We consider $\Flag_n:=\left\{0\in V_1\subset\dots\subset V_{n-1}\subset \mathbb R^n: \dim V_k=k\right\}$, the projective variety of full flags of vector subspaces of $\mathbb R^n$.
We embed $\Flag_n$ into the product of Grassmannians and
combine with the Pl\"ucker embeddings for every $k$. 
So $\Flag_n\hookrightarrow \mathbb P^{\binom{n}{1}-1}\times \dots \times \mathbb P^{\binom{n}{n-1}-1}$ is the vanishing set of an ideal in the 
multigraded polynomial ring $S:=\mathbb R[p_J: J\subset [n]]$ with grading given by $\deg(p_J):=\omega_{\vert J\vert}\in \mathbb Z^{n-1}$ (where $\{\omega_\ell: 1\le \ell\le n-1\}$ denotes to standard basis).
The variables $p_J$ are called Pl\"ucker variables and
the (multi-)homogeneous ideal $J_n$ defining $\Flag_n$ is generated by Pl\"ucker relations (see e.g. \cite[\S1.2]{Feigin_degFlag} for their precise form).

\begin{Lemma}\thlabel{lem:sign in 3-term plückers}
Consider $J\subset[n]$ of cardinality less or equal to $n-4$.
Choose $i,j,k\not \in J$ and $1\le i<j<k\le n$. 
Then the following is a {\bf three term Pl\"ucker relation}
\begin{eqnarray}\label{eq:3 term rel}
    R_{i,j,k}^J:= p_{\{i\}\cup J}p_{\{j,k\}\cup J} - p_{\{j\}\cup J}p_{\{i,k\}\cup J} + p_{\{k\}\cup J}p_{\{i,j\}\cup J}\in I_n.
\end{eqnarray}
\end{Lemma}

\begin{Lemma}
The Plücker ideal $J_n$ is invariant under the action of the symmetric group on $\mathbb R[p_J:j\subset[n]]$ given for $\sigma \in S_n$ and $J=\{j_1,\dots,j_k\}\subset [n]$ by
\[
\sigma(p_{\{j_1,\dots,j_k\}})=(-1)^{\text{sgn}(J,\sigma)}p_{\{\sigma(j_1),\dots,\sigma(j_k)\}},
\]
where $\text{sgn}(J,\sigma)$ is the number of transpositions necessary to order $\sigma(j_1),\dots,\sigma(j_k)$ increasingly.
\end{Lemma}

\begin{Definition}
Let $J_n\subset \mathbb R[p_J:j\subset [n]]$ be the Plücker ideal of the flag variety. The {\bf tropical totally positive flag variety} is $\trop^+(\Flag_n):=\trop^+(J_n)$.
\end{Definition}

\begin{Example}
Consider $\Flag_3$ defined by $J_3=(p_{1}p_{23}-p_2p_{13}+p_3p_{12})\subset\mathbb R[p_1,p_2,p_3,p_{12},p_{13},p_{23}]$.
Its tropicalization is a one dimensional fan in $\mathbb R^6/\mathcal L_{J_3}$ whose maximal cones are generated by the cosets of $w_1=(0,1,1,0,0), w_2=(1,0,1,0,0,0), w_3=(1,1,0,0,0,0)$. The corresponding initial ideals are
$\init_{w_1}(J_3)=(-p_{2}p_{13}+p_{3}p_{12})$,  $\init_{w_2}(J_3)=(p_{1}p_{23}+p_{3}p_{12})$, and $\init_{w_3}(J_3)=(p_{1}p_{23}-p_{2}p_{13})$.
Notice that $\init_{w_2}(J_3)$ is not totally positive, hence $\trop^+(J_3)$ consists of only two one dimensional cones in $\mathbb R^6/\mathcal L_{J_3}$ corresponding to $w_1$ and $w_3$.
\end{Example}

\subsection{The tropical totally positive flag variety for $n=4$}
In this section we consider the tropical flag variety $\trop(\Flag_4)$ with two fan structures: one induced by the Gröbner fan of the Plücker ideal $J_4$, the other induced by the Gröbner fan of an extended ideal generated by the exchange relations of the corresponding cluster algebra, see \S\ref{sec:extended F4}.
A more extensive discussion about $\trop^+(\Flag_4)$ and its different fan structures can be found here \cite{BEW_flag-positroids}

In \cite{BLMM} the authors computed the tropicalization of the Plücker ideals $J_4$ and $J_5$.

\begin{Theorem}[\cite{BLMM}]
The tropical flag variety $\trop(\Flag_4)$ with its fan structure induced by the Gröbner fan of the Plücker ideal $J_4$ is a six dimensional fan $\trop(J_4)\subset \mathbb R^{14}/\mathcal L_{J_4}$  with 78 maximal cones. 
The initial ideals of 72 maximal cones are binomial and prime and split into four orbits with respect to the action of the symmetric group and $\mathbb Z_2$.
\end{Theorem}

We complement their findings by the following observations regarding the totally positive part of the tropical flag variety.
Let 
$\{e_1,e_2,e_3,e_4,e_{12},e_{13},e_{14},e_{23},e_{24},e_{34},e_{123},e_{124},e_{134},e_{234}\}$ 
denote the standard basis for $\mathbb R^{14}$.
The 20 rays of $\trop(J_4)/\mathcal L_{J_4}$ are generated by the standard basis elements and
\[\begin{split}
r_{12}:=e_{12}+e_{123}+e_{124},\quad r_{13}:=e_{13}+e_{123}+e_{134},\quad r_{14}:=e_{14}+e_{124}+e_{134},\\ r_{23}:=e_{23}+e_{123}+e_{234},\quad r_{24}:=e_{24}+e_{124}+e_{234},\quad r_{34}:=e_{34}+e_{134}+e_{234}.    
\end{split}
\]
$\trop^+(J_4)$ has nine rays $e_1, e_{12}, e_{123}, e_4, e_{34}, e_{234}, r_{12},r_{23}, r_{34}$ and 14 maximal cones that form a polyhedral fan dual to the associahedron, see Figure~\ref{fig:trop+F4} on the left (here, for example, the top vertex corresponds to the maximal cone whose rays are $e_1,e_{34},e_{123}$). 
\begin{figure}
    \centering
 \begin{tikzpicture}[scale=.7]
\draw(0,6) -- (2.5,4.5) -- (1.5,2) -- (-1.5,2) -- (-2.5,4.5) -- (0,6);
\node at (-.5,5) {$e_{34}$};
\draw(0,0) -- (2.5,.5) -- (1.5,2) -- (-1.5,2) -- (-2.5,.5) -- (0,0);
\node at (.5,1) {$e_{12}$};
\draw[dashed,fill,blue,opacity=.1] (0,6) -- (2.5,4.5) -- (3.5,2.5) -- (2,3) -- (0,4.5) -- (0,6);
\node[blue] at (1.5,4) {$e_{123}$};
\draw[dashed] (0,6) -- (2.5,4.5) -- (3.5,2.5) -- (2,3) -- (0,4.5) -- (0,6);
\draw[dashed,fill,red,opacity=.1] (0,6) -- (-2.5,4.5) -- (-3.5,2.5) -- (-2,3) -- (0,4.5) -- (0,6);
\draw[dashed] (0,6) -- (-2.5,4.5) -- (-3.5,2.5) -- (-2,3) -- (0,4.5) -- (0,6);
\node[red] at (-1.5,4) {$e_{1}$};
\draw[dashed,fill,green,opacity=.1] (0,0) -- (2.5,.5) -- (3.5,2.5) -- (2,3) -- (0,1.5) -- (0,0);
\draw[dashed] (0,0) -- (2.5,.5) -- (3.5,2.5) -- (2,3) -- (0,1.5) -- (0,0);
\node[green] at (1.5,1.5) {$e_{4}$};
\draw[dashed,fill,yellow,opacity=.2] (0,0) -- (-2.5,.5) -- (-3.5,2.5) -- (-2,3) -- (0,1.5) -- (0,0);
\draw[dashed] (0,0) -- (-2.5,.5) -- (-3.5,2.5) -- (-2,3) -- (0,1.5) -- (0,0);
\node[brown] at (-1.25,1.5) {$e_{234}$};
\draw (-2.5,.5) -- (-1.5,2) -- (-2.5,4.5) -- (-3.5,2.5) -- (-2.5,.5);
\node at (-2.5,2) {$r_{34}$};
\draw (2.5,.5) -- (1.5,2) -- (2.5,4.5) -- (3.5,2.5) -- (2.5,.5);
\node at (2.5,2) {$r_{12}$};
\draw[dashed,fill,gray,opacity=.1] (0,4.5) -- (2,3) -- (0,1.5) -- (-2,3) -- (0,4.5);
\node[gray] at (0,3) {$r_{23}$};
\draw[dashed] (0,4.5) -- (2,3) -- (0,1.5) -- (-2,3) -- (0,4.5);

\draw[fill,white] (1.5,2) circle (.1cm);
\draw[fill,white] (-1.5,2) circle (.1cm);
\draw (1.5,2) circle (.1cm);
\draw (-1.5,2) circle (.1cm);

\begin{scope}[xshift=10cm]
\draw(0,6) -- (2.5,4.5) -- (0,3) -- (-2.5,4.5) -- (0,6);
\node at (-.5,5) {$e_{34}$};
\draw(0,0) -- (2.5,.5) -- (0,1.245) -- (-2.5,.5) -- (0,0);
\node at (.5,.5) {$e_{12}$};
\draw[dashed,fill,blue,opacity=.1] (0,6) -- (2.5,4.5) -- (3.5,2.5) -- (2,3) -- (0,4.5) -- (0,6);
\node[blue] at (1,4.5) {$e_{123}$};
\draw[dashed] (0,6) -- (2.5,4.5) -- (3.5,2.5) -- (2,3) -- (0,4.5) -- (0,6);
\draw[dashed,fill,red,opacity=.1] (0,6) -- (-2.5,4.5) -- (-3.5,2.5) -- (-2,3) -- (0,4.5) -- (0,6);
\draw[dashed] (0,6) -- (-2.5,4.5) -- (-3.5,2.5) -- (-2,3) -- (0,4.5) -- (0,6);
\node[red] at (-2,3.5) {$e_{1}$};
\draw[dashed,fill,green,opacity=.1] (0,0) -- (2.5,.5) -- (3.5,2.5) -- (2,3) -- (0,1.5) -- (0,0);
\draw[dashed] (0,0) -- (2.5,.5) -- (3.5,2.5) -- (2,3) -- (0,1.5) -- (0,0);
\node[green] at (1.5,1.5) {$e_{4}$};
\draw[dashed,fill,yellow,opacity=.2] (0,0) -- (-2.5,.5) -- (-3.5,2.5) -- (-2,3) -- (0,1.5) -- (0,0);
\draw[dashed] (0,0) -- (-2.5,.5) -- (-3.5,2.5) -- (-2,3) -- (0,1.5) -- (0,0);
\node[brown] at (-1.5,1.5) {$e_{234}$};
\draw (-2.5,.5) -- (0,1.245) -- (2.5,.5) -- (3.5,2.5) -- (2.5,4.5)-- (0,3) -- (-2.5,4.5) -- (-3.5,2.5) -- (-2.5,.5);
\draw (0,1.245) -- (0,3);
\node at (-2.5,2) {$r_{34}$};
\node at (2.5,2) {$r_{12}$};
\draw[dashed,fill,gray,opacity=.1] (0,4.5) -- (2,3) -- (0,1.5) -- (-2,3) -- (0,4.5);
\node[gray] at (0,3.5) {$r_{23}$};
\draw[dashed] (0,4.5) -- (2,3) -- (0,1.5) -- (-2,3) -- (0,4.5);
\end{scope}

(1.5,2)
\end{tikzpicture}
    \caption{The fan structure of $\trop^+(\Flag_4)$ induced by the Gröbner fan of the Plücker ideal is dual to the associahedron on the left. The fan structure induced by the Gröbner fan of the extended ideal is dual to the associahedron on the right. Colored faces are on the \emph{back}.}
    \label{fig:trop+F4}
\end{figure}
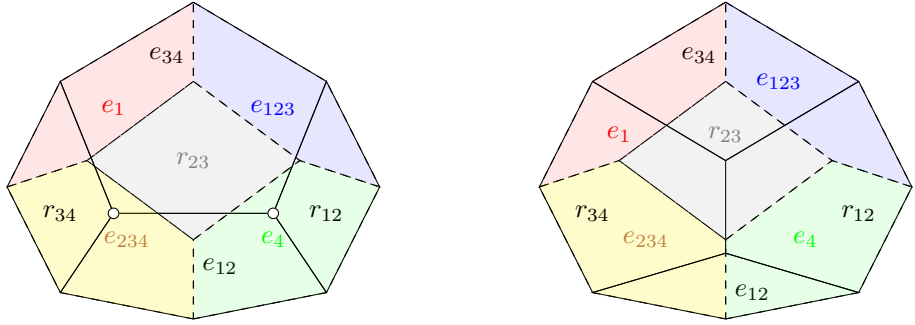
A complete list of all cones in terms of there rays generators is provided below:
\[
\begin{split}
\{e_1, r_{34}, e_{34}\}= C_0,  
\{e_1, r_{34}, e_{234}\}=C_1, 
\{e_1, e_{234}, r_{23}\}=C_3, 
\{e_1, e_{34}, e_{123}\}=C_8, 
\{e_1, r_{23}, e_{123}\}=C_{14}, \\
\{r_{34}, e_{34}, e_{12}\}=C_{17}^\dagger,  
\{r_{34}, e_{234}, e_{12}\}=C_{18},
\{e_{234}, r_{23}, e_4\}=C_{24},
\{e_{34}, e_{123}, r_{12}\}=C_{36}, 
\{r_{23}, e_{123}, e_4\}=C_{44}, \\
\{e_{34}, e_{12}, r_{12}\}=C_{51}^\dagger, 
\{e_{234}, e_{12}, e_4\}=C_{53}, 
\{e_{123}, r_{12}, e_4\}=C_{71},    
\{e_{12}, r_{12}, e_4\}=C_{77}. 
\end{split}
\]
The $C_i$'s are the notation from \cite{BLMM} as they can be found on the homepage \url{https://www.matem.unam.mx/~lara/tropflag/}.
All initial ideals associated with the maximal cones are generated by binomials, all but two are prime. The $\dagger$ in the above list indicated the non-prime cones.
In the associahedron the corresponding vertices are labeled by white circles.

\subsubsection{Changing the embedding}\label{sec:extended F4}
In \cite{BLMM} an algorithmic procedure is suggested to increment the number of maximal prime cones in the tropicalization of an ideal. 
The basic idea is to add generators to the polynomial ring with well-chosen relations that will provide the "missing elements" in non-prime initial ideals.
We apply their procedure and add a single generator $x$ to the polynomial ring $\mathbb R[p_J:J\subset [4]]$.
It corresponds to the only non-Plücker cluster variable of the cluster algebra.
The (exchange) relations between $x$ and the Plücker variables are
\[
\begin{split}
xp_{23}-p_{12}p_{234}p_{3}-p_{2}p_{34}p_{123},\
xp_{24}-p_{4}p_{12}p_{234}-p_{2}p_{34}p_{124}, \
xp_{13}-p_{3}p_{12}p_{134}-p_{1}p_{34}p_{123},\\
xp_{14}-p_{4}p_{12}p_{134}-p_{1}p_{34}p_{124}, \
p_{124}p_{3}-x-p_{4}p_{123},\
p_{2}p_{134}-x-p_{1}p_{234}.
\end{split}
\]
Let $I^{\rm ex}\subset \mathbb R[x,p_J:J\subset [4]]$ be the ideal generated by the Plücker relations together with the additional six relations above. 
We fix $\deg(x)=2,\deg(p_J)=1$ so that $I^{\rm ex}$ is homogeneous with respect to an $\mathbb N$-grading and verify that it is saturated.
In fact, only the relations of degree two are necessary generators of the ideal.
As $\mathbb R[x,p_J:J\subset[4]]/I^{\rm ex}\cong \mathbb R[p_J:J\subset[4]]/J_4$ the positive part of the tropicalization $\trop(I^{\rm ex})$ also induces a fan structure on $\trop^+(\Flag_4)$.
This Gröbner fan structure on $\trop(I^{\rm ex})$ is simplicial, contains a six dimensional linear subspace $\mathcal L_{I^{\rm ex}}$ and $\trop(I^{\rm ex})/\mathcal L_{I^{\rm ex}}$  has f-vector $(1, 25, 105, 105)$.
The natural projection $\mathbb R^{15}\to \mathbb R^{14}$ maps $\trop(I^{\rm ex})$ to $\trop(J_4)$.
We are interested in $\trop^+(I^{\rm ex})$ and its image under the projection.
Let $e_x,e_1,\dots,e_{234}$ denote the standard basis of $\mathbb R^{15}$.
Then the rays of $\trop^+(I^{\rm ex})$ are 
$e_1,\ e_{12},\ e_{123},\ e_4,\ e_{34},\ e_{234},\ r_{12}+e_x,\ r_{23},\ r_{34}+e_x$, where we slightly abuse notation and identifying $r_{ij}$ with their corresponding linear combination in the standard basis.
The projection from $\trop^+(I^{\rm ex})$ to $\trop^+(\Flag_4)$ however is \emph{not} a map of fans: while $\{r_{12},e_{12},e_{34}\}$ and $\{r_{34},e_{12},e_{34}\}$ are maximal cones in $\trop^+(J_4)$, in $\trop^+(I^{\rm ex})$ we find that the corresponding sets $\{r_{12},e_{12},r_{34}\}$ and $\{r_{34},e_{12},r_{34}\}$ generate maximal cones.
The complete list of cones is as follows
\[
\begin{split}
\{e_1, r_{34}, e_{34}\}= C_0,  
\{e_1, r_{34}, e_{234}\}=C_1, 
\{e_1, e_{234}, r_{23}\}=C_3, 
\{e_1, e_{34}, e_{123}\}=C_8, 
\{e_1, r_{23}, e_{123}\}=C_{14}, \\
\{r_{34}, e_{34}, r_{12}\}=\mathcal C_{17}^\dagger,  
\{r_{34}, e_{234}, e_{12}\}=C_{18},
\{e_{234}, r_{23}, e_4\}=C_{24},
\{e_{34}, e_{123}, r_{12}\}=C_{36}, 
\{r_{23}, e_{123}, e_4\}=C_{44}, \\
\{r_{34}, e_{12}, r_{12}\}=\mathcal C_{51}^\dagger, 
\{e_{234}, e_{12}, e_4\}=C_{53}, 
\{e_{123}, r_{12}, e_4\}=C_{71},    
\{e_{12}, r_{12}, e_4\}=C_{77}. 
\end{split}
\]
The rays of $\trop^+(I^{\rm ex})$ generate a unique maximal cone in $\GF(I^{\rm ex})$ whose initial monomial ideal is the Stanley--Reisner ideal of  $\trop^+(I^{\rm ex})$.
This observation fits the findings of \cite{BMN} where the authors study the interaction of the cluster structure of $\Gr(2,n)$ and $\Gr(3,6)$ and its relation to the Gröbner theory of defining ideals.
As in {\it loc.cit.} we obtain a bijection between rays of $\trop^+(I^{\rm ex})$ and mutable cluster variables:
$(p_{13},r_{23}),(p_{24},r_{34}),(x,r_{12}),(p_2,e_{12}),(p_{134},e_{234}),(p_{23},e_{123}),(p_{124},e_{34}),(p_3,e_4),(p_{14},e_1)$.
\subsection{FFLV degeneration}

We summarize the construction of the toric degeneration of the flag variety into the toric variety associated with the FFLV polytope.
Instead of relying on the original reference \cite{FFL_PBWtypeA}, we chose the language of valuations as introduced in \cite{FFL_birat}.

Consider the Lie algebra $\mathfrak{sl}_{n}$ and a Cartan subalgebra $\mathfrak h$ inducing a decomposition $\mathfrak{sl}_{n}=\mathfrak{n}\oplus \mathfrak h\oplus \mathfrak n^-$.
Let $\omega_1,\dots,\omega_{n-1}$ denote the fundamental weights and $R^+:=\{\epsilon_i-\epsilon_j: 1\le i<j\le n\}$ be the set of positive roots.
For a dominant integral weight $\lambda$ let $V_\lambda$ be the corresponding irreducible highest weight representation of $\mathfrak{sl}_{n}$.
Fix a sequence $\mathcal S=(\epsilon_{i_1}-\epsilon_{j_1},\dots,\epsilon_{i_N}-\epsilon_{j_N})$ containing all positive roots such that $\epsilon_i-\epsilon_j$ appears before $\epsilon_{i'}-\epsilon_{j'}$ whenever $j-i>j'-i'$.
Recall, that for every $\epsilon_i-\epsilon_j\in R^+$ we have an element $f_{i,j}\in \mathfrak n^-$.
A Pl\"ucker coordinate $p_J$ with $J=\{j_1,\dots,j_\ell\}$ is the basis element of $V_{\omega_\ell}^*=(\bigwedge^\ell k^n)^*$ dual to $e_{j_1}\wedge\dots\wedge e_{j_\ell}\in \bigwedge^\ell k^n$ (here $\{e_j\}_j$ denoted the standard basis of $k^n$).
The $f_{a,b}$ operate on $\bigwedge^\ell k^n$.
Let $\prec$ be the homogeneous right lexicographic order in $\mathbb Z^N$, where $N=\dim_k\Flag_n$.
To an element $a\in \mathbb Z^N$ we associate a monomial ${\bf f}_{\mathcal S}^{a}:=f_{i_1,j_1}^{a_1}\cdots f_{i_N,j_N}^{a_N}\in U(\mathfrak n^-)$ in the universal enveloping algebra of $\mathfrak n^-$.
For every Pl\"ucker coordinate $p_J$ with $J=\{j_1,\dots,j_\ell\}$ we define
\[
m_J:=\text{min}{}_{\prec}\{ a\in \mathbb Z^N : {\bf f}_S^a(e_1\wedge\dots \wedge e_\ell)=e_{j_1}\wedge\dots\wedge e_{j_\ell} \}.
\]
By \cite[\S13]{FFL_birat} this extends to a valuation $\hat\val_{\mathcal S}:k[SL_n/U]\to \mathbb Z^N\times \Lambda$ with $\hat\val_{\mathcal S}(p_J)=(m_J,\omega_\ell)$, where $U\subset SL_{n}$ is a maximal unipotent subgroup whose Lie algebra is $\mathfrak n$.
The Pl\"ucker coordinates form a Khovanskii basis for this valuation (see e.g. \cite{B-quasival}) and by \cite[Theorem 13.3]{FFL_birat} it induced a toric degeneration of the flag variety into the projective toric variety associated with the {\bf FFLV polytope}:
\[
P:=\text{conv}(m_J: J\subset[n]).
\]
The toric variety of $P$ can also be described as the vanishing set of an initial ideal of $I_n$. 
Set $K:=\binom{n}{1}+\dots +\binom{n}{n-1}$.
We define $M\in \mathbb Z^{N\times K}$ be the matrix with columns $\val(p_I), I\subset [n]$.
By \cite[Theorem 1]{B-quasival} we have $\init_M(I_n)=( \init_M(R^k_{L,J}): \forall k,J,L )=:I_{\rm FFLV}$
and the toric variety defined by $P$ is the normalization of $V(\init_M(I_n))$.
Let $\{e_{1},\dots,e_N\}$ be the standard basis of $\mathbb Z^N$ where $e_p$ corresponds to $\epsilon_{i_p}-\epsilon_{j_p}$ in the sequence $\mathcal S$.
Fang, Fourier and Reineke define a {\bf degree map} $\deg:\mathbb Z^N\to \mathbb Z$ by $e_p\mapsto (j_p-i_p+1)(n-j_p+i_p)$ \cite{FFR15}.
This map can be used to project $M$ to a weight vector in $\mathbb Z^N$ by applying $\deg$ to its columns in the sense of \cite[Lemma 3]{B-quasival}.
Therefore, $I_{\rm FFLV}$ is the initial ideal with respect to a unique maximal prime cone in $\trop(\Flag_n)$ which we call $C_{\rm FFLV}$. 
Explicit inequalities for $C_{\rm FFLV}$ are available in \cite{FFFM_PBW_tropFlag}.

\begin{Lemma}\label{lem:val FFLV pluecker}
Let $L=\{l_1,\dots,l_k\}\subset [n]$ and assume $l_1<\dots<l_s\le k< l_{s+1}<\dots <l_k$. 
Further, let $\{p_{s+1}<\dots <p_{k}\}=[k]\setminus \{l_1,\dots,l_s\}$ and $S=(\epsilon_{i_1}-\epsilon_{j_1},\dots,\epsilon_{i_N}-\epsilon_{j_N})$ be a sequence with good ordering. Then for $1\le r\le N$ and $s+1\le t\le k$ we have
\[
(m_L)_{i_r,j_r}=\left\{\begin{matrix}
1 &\text{ if } (i_r,j_r)=(p_t,l_t) \text{ for some } s+1\le t\le k\\
0 &\text{ otherwise}. 
\end{matrix}\right.
\]
\end{Lemma}
\begin{proof}
In the definition of $m_L$, note that every operator $f_{i_r,j_r}$ for which $(m_L)_{i_r,j_r}$ is non-zero acts on one exactly one of the $e_{l_i}$'s.
In particular, $m_L$ encodes which $e_{l_i}$ get send to which $e_i$ for $1\le i\le k$.
As $\prec$ is an homogeneous order, every $m_J$ has as few non-zero entries as possible.
Hence, if $l_i\le k$ the corresponding $e_{l_i}$ gets send to itself.
The good ordering of $S$ in combination with the right lexicographic order implies that if $l_t<l_q$ then $e_{l_t}$ gets send to $e_{i_t}$ and $e_{l_q}$ to $e_{i_q}$ with $1\le i_t<i_q\le k$.
\end{proof}

\begin{Corollary}\label{cor:3-term initial}
For $[I]\subset [n]$ with $\vert I\vert=j-2$ and $1\le i<j<k\le n$ we have
\[
\init_M(R_{i,j,k}^J) = p_{\{i\}\cup J}p_{\{j,k\}\cup J} + p_{\{k\}\cup J}p_{\{i,j\}\cup J}.
\]
\end{Corollary}
\begin{proof}
It suffices to verify that $m_{\{i\}\cup J}+m_{\{j,k\}\cup J}=m_{\{k\}\cup J}+m_{\{i,j\}\cup J}$ which is a consequence of Lemma~\ref{lem:val FFLV pluecker}.
\end{proof}

\begin{Corollary}\label{cor:3-term initials ijk}
Let $1\le i<j<k\le n$ and recall the relations defined in \eqref{eq:3 term rel}. Then for $J=\varnothing$
\[
\init_M(R_{i,j,k})=\left\{\begin{matrix}
p_{i}p_{jk}+p_kp_{ij} & \text{ if } i=1,j=2\\
p_{i}p_{jk}-p_{j}p_{ik} &\text{ if } i=1,j\not =2 \text{ or } i>2\\
-p_jp_{ik}+p_kp_{ij} & \text{ if } $i=2$.
\end{matrix}\right.
\]
\end{Corollary}

\begin{Example}\label{exp:initial forms for proof}
Let $n\ge 5$. Then by Corollary~\ref{cor:3-term initials ijk} we have the following initial forms in $I_{\rm FFLV}$:
\begin{eqnarray*}
\init_M(R_{1,2,3}) &=& p_1p_{23}+p_3p_{12}, \quad
\init_M(R_{1,2,4}) = p_1p_{24}+p_4p_{12}, \\
\init_M(R_{1,3,4}) &=& p_1p_{34}-p_3p_{14}, \quad
\init_M(R_{1,4,5}) = p_1p_{45}-p_4p_{15}.
\end{eqnarray*}
Notice that the first initial forms imply that $I_{\rm FFLV}$ is not totally positive, hence it does not correspond to a maximal prime cone in $\trop^+(\Flag_n)$.
By Corollary~\ref{cor:3-term initial} and Lemma~\ref{lem:val FFLV pluecker} we get $\init_M(R_{2,3,4}^1) = p_{12}p_{134}+p_{14}p_{123}$,
$\init_M(R_{2,4,5}^1) = p_{12}p_{145}+p_{15}p_{124}$ and $\init_M(R_{1,2,3}^4) = p_{14}p_{234}-p_{24}p_{134}$.
\end{Example}

The main result of this section can now be stated as follows.

\begin{Theorem}\label{thm:FFLV not pos}
For $n\ge 5$ the symmetric group orbit of the maximal prime cone $C_{\rm FFLV}\subset \trop(\Flag_n)$ does not intersect the totally positive part $\trop^+(\Flag_n)$.
Said differently, there is no $\sigma\in S_n$ such that $\sigma(I_{FFLV})$ is totally positive.
\end{Theorem}

\begin{Example}
For $n=4$ the theorem is false. Below on the left is the FFLV-ideal and on the right its image after applying the permutation $\sigma=(1342)\in S_4$ which is totally positive:
\begin{eqnarray*}
\left(
\begin{matrix}
p_{24}p_{134} - p_{14}p_{234}, & p_{14}p_{123} + p_{12}p_{134},\\
p_{4}p_{23} - p_{3}p_{24}, &p_{3}p_{12} + p_{1}p_{23},\\
p_{13}p_{24} - p_{12}p_{34} &p_{34}p_{123} + p_{13}p_{2},\\
p_{4}p_{13} + p_{1}p_{34}, &p_{24}p_{123} + p_{12}p_{234},\\
p_{4}p_{12} + p_{1}p_{24}, &p_{4}p_{123} - p_{1}p_{234}.
\end{matrix}\right),
\left(
\begin{matrix}
p_{34}p_{123} - p_{23}p_{134}, & p_{24}p_{123} - p_{23}p_{124}, \\
p_{3}p_{14} - p_{1}p_{34}, &p_{2}p_{14} - p_{1}p_{24}, \\
p_{13}p_{24} - p_{12}p_{34}, &p_{13}p_{124} - p_{12}p_{134}, \\
p_{3}p_{12} - p_{2}p_{13}, &p_{34}p_{124} - p_{24}p_{134}, \\
p_{3}p_{24} - p_{2}p_{34}, &p_{3}p_{124} - p_{2}p_{134}.
\end{matrix}\right)
.
\end{eqnarray*}
\end{Example}

\begin{proof}[Proof of Theorem~\ref{thm:FFLV not pos}]
Assume there exists $\sigma\in S_n$ such that $\sigma(I_{FFLV})$ is totally positive. 
There are six cases for the the image of the sequence $(1,2,3)$ under $\sigma$ that we distinguish. 
In each case we show that $\sigma(I_{\rm FFLV})$ contains a nonzero element of $\mathbb R_{\ge 0}[p_J:J\subset [n]]$ which leads to a contradiction.
We will use the initial forms computed in Example~\ref{exp:initial forms for proof}. 
Let $1\le i<j<k\le n$, denote $\sigma(4)=x$ and $\sigma(5)=y$.
\begin{itemize}
    \item[$\sigma((1,2,3))=(i,j,k)$] In this case 
    $\sigma(p_{3}p_{12}+p_1p_{23})=p_{k}p_{ij}+p_ip_{kj}\in \sigma(I_{\rm FFLV})$,
    a contradiction.
    \item[$\sigma((1,2,3))=(i,k,j)$] Then $\sigma(p_{24}p_{134}-p_{14}p_{234}) = \pm (p_{kx}p_{ijx} + p_{ix}p_{jkx}) \in \sigma(I_{\rm FFLV})$,
    a contradiction.
    \item[$\sigma((1,2,3))=(j,i,k)$] Compute $\sigma(p_1p_{24}+p_4p_{12})=p_jp_{ix}-p_xp_{ij}$, which implies $i<x$. Further, $\sigma(p_1p_{34}-p_3p_{14})=p_jp_{kx}-p_kp_{jx}$ implies either $k<x$ or $j>x$. But if $k<x$ then $\sigma(p_{14}p_{123}+p_{12}p_{134})= -p_{jx}p_{ijk}-p_{ij}p_{jkx}$, a contradiction. So we have $i<x<j<k$. Repeating the computation for $5$ insted of $4$ we get $i<y<j<k$.
    Next we consider $\sigma(p_1p_{45}-p_4p_{15})=-p_jp_{yx}+p_xp_{yj}$, which implies $y<x$. However, in this case $\sigma(p_{12}p_{145}-p_{14}p_{125})= p_{ij}p_{yxj}+p_{xj}p_{iyj}$, a contradiction.
    \item[$\sigma((1,2,3))=(j,k,i)$] We compute $\sigma(p_1p_{24}+p_4p_{12})=-p_jp_{kx}+p_xp_{jk}$ which implies $x<k$. Then $\sigma(p_1p_{34}-p_3p_{14})=p_jp_{ix}-p_ip_{jx}$ so we need either $j<x$ or $i>x$ for the assumption to hold.
    But if $i<x$ then $\sigma(p_{14}p_{123}+p_{12}p_{134})=-p_{xj}p_{ijk}-p_{jk}p_{xij}$, a contradiction. So $i<j<x<k$. Similarly, repeating the computations for $5$ instead of $4$ we obtain $i<j<y<k$.
    Next consider $\sigma(p_1p_{45}+p_4p_{15})=-p_jp_{yx}+p_xp_{iy}$ which implies $y<x$. However in this case we have $\sigma(p_3p_{45}-p_4p_{35})=-p_ip_{yx}-p_xp_{iy}$, a contradiction.
    \item[$\sigma((1,2,3))=(k,i,j)$] We have $\sigma(-p_{24}p_{134}+p_{14}p_{234})=\pm(p_{ix}p_{jkx} + p_{kx}p_{ijx}) \in \sigma(I_{\rm FFLV})$,
    a contradiction.          
    \item[$\sigma((1,2,3))=(k,j,i)$] In this case $\sigma(p_{3}p_{12}+p_1p_{23})=-p_{i}p_{jk}-p_kp_{ij}\in \sigma(I_{\rm FFLV})$,
    a contradiction.
\end{itemize}
\end{proof}

Combining Theorem~\ref{thm:FFLV not pos} and Theorem~\ref{thm:main} an immediate consequence is Corollary~\ref{cor:FFLV not gv sym}.

\footnotesize{
\newcommand{\etalchar}[1]{$^{#1}$}

}

\end{document}